\documentclass[a4paper, leqno,11pt]{article}
\usepackage[left = 2.5cm, right = 2.5cm, top = 2.5cm, bottom = 2.5cm]{geometry}
\usepackage{amssymb,amsmath,amsthm,mathrsfs,enumerate,graphicx, color}
\usepackage[pdfpagelabels,colorlinks,linkcolor=blue,citecolor=black,urlcolor=blue]{hyperref}
\usepackage{esint}

\makeatother
\newtheorem{theorem}{Theorem}[section]
\newtheorem{lemma}[theorem]{Lemma}

\newtheorem{proposition}[theorem]{Proposition}
\newtheorem{corollary}[theorem]{Corollary}
\newtheorem{remark}[theorem]{Remark}

\newcommand{\R}{\mathbb{R}}

\title{H\"older continuity of solutions for a class of drift-diffusion equations}

\author{Quoc-Hung Nguyen \thanks{E-mail address: qhnguyen@amss.ac.cn, Academy of Mathematics and Systems Science, Chinese Academy of Sciences, Beijing, 100190, China.}, Yannick Sire\thanks{E-mail address: ysire1@jhu.edu, Department of Mathematics, Johns Hopkins University
		3400 N. Charles Street, Baltimore, MD 21218, USA.} and Le Xuan Truong \thanks{E-mail address: lxuantruong@ueh.edu.vn, University of Economics Ho Chi Minh City (UEH), Ho Chi Minh City, Vietnam.}}

\begin{document}

\date{}
\maketitle

\begin{abstract}
We provide several regularity results for non-homogeneous drift-diffusion equations with applications to general dissipative SQG. Our results unify in a rather simple way several previously known results. We build the estimates on an algebraic identity (for the refined energy argument) which relates any integral operator with pure powers of the laplacian. 
\end{abstract}

\tableofcontents

\section{Introduction}	

We consider drift-diffusion equations involving rather general integral operators (with bounded {\sl elliptic} measurable kernels). We aim at proving H\"older regularity of solutions whenever the right hand side has enough integrability. We apply also our results to the dissipative SQG equation, i.e. whenever the divergence-free vector-field in the transport term is related in a non-trivial way to the solution of the equation by a Darcy law.  

The standard dissipative SQG equation has been introduced by Constantin, Majda and Tabak (see \cite{CMT} for instance) as a toy model for regularity in three-dimensional flows. Besides qualitative properties of the equations, such as long-time behaviour, a major issue is the regularity of weak solutions. This is the issue we address here for a general version of drif-diffusion equations. 

\vspace{0.2cm}
The standard {\sl critical} dissipative SQG equation is 
\begin{align}\label{standardSQG}
&\partial_t u+{\bf B} \cdot \nabla u+(-\Delta)^{1/2} u =0~~\text{in}~~\mathbb{R}^2 \times\mathbb{R},
\end{align}
where ${\bf B}$ is divergence free and related to $u$ by ${\bf B}=\nabla^\perp (-\Delta)^{-1/2}u$. The regularity of (weak) solutions has been established in \cite{CaVa} for data satisfying an {\sl energy type} inequality (see also \cite{KNV,KN}). Several generalizations of this equation can be considered, where the Fourier multiplier $(-\Delta)^\frac12$ is replaced by $(-\Delta)^s$, yielding sub critical whenever $s>\frac12$ (see \cite{CF}) and super critical whenever $s<\frac12$ (see \cite{Cons-Wu,consWu}) equations. One could also consider the case of bounded domains such as in \cite{Con-Igna}. Several tools have been developed to handle drift-diffusion PDEs, whenever the diffusion is given by a Fourier multiplier. The mathematical difficulties substantially increase when one considers other types of integral operators like the ones described below, which are not Fourier multipliers. 

\vspace{0.2cm}
In the present work, we consider the following general drift-diffusion equation  for $d \geq 2$
\begin{align}\label{eq1}
&\partial_t u + {\bf B}.\nabla u+\mathcal{L}_t u = g ~~\text{in}~~\mathbb{R}^d \times\mathbb{R}=\mathbb{R}^{d+1},
\end{align}
where $g\in L^1(\mathbb{R}^{d+1})$ and $\mathbf{B} : \mathbb{R}^d \times \mathbb{R} \to \mathbb{R}^d$ is a vector field satisfying $\operatorname{div}(\mathbf{B})=0$. Furthermore, the integral operator $\mathcal{L}_t$ ($t \in \mathbb{R}$) is defined by  
\begin{align*}
\mathcal{L}_t  u(x) = \text{P.V.}\int_{\mathbb{R}^d}(u(x)-u(y))\mathbf{K}_t(x,y)dy,
\end{align*}
where the kernel has some structure we will discuss below but is of homogeneity $-2s$ similarly to the fractional laplacian (hence ensuring some ellipticity). Notice that parabolic equations in such generality (for the diffusion) have been considered for instance in \cite{CCV}. For the general operator we are interested in, there is of course no extension as a local operator in one more dimension. This fact, already present in \cite{CCV}, introduces several non trivial difficulties and one has to develop a somehow {\sl intrinsically} nonlocal/integral method to obtain H\"older regularity. Contrary to \cite{CCV}, we actually {\sl use} the extension in one more dimension by relating energy estimates (the crucial part of De Giorgi's argument) to the ones obtained for pure powers of the laplacian. This allows in particular to simplify the proof in \cite{CCV}. We focus on the most difficult case of the critical/supercritical regime $s\leq 1/2$.

\vspace{0.2cm}
As mentioned previously, in order to tackle the case of general integral diffusion as described above, we take here a different route from for \cite{CCV}, bridging in a natural way the $s-$harmonic extension of a function $u$ to its Gagliardo semi-norms, which control from above and below the kernel ${\bf K}_t$. This approach is new  and we believe could be used successfully in other contexts. We now explain this link since this is the core of the argument. In proving H\"older regularity by De Giorgi's method, the crucial step is to go from an $L^2$ type estimate to local boundedness, via a special energy inequality obtained by testing the equation with suitable truncations of the unknown. Using a trick we describe below, one can relate the purely nonlocal energy inequality coming from $\mathcal L_t$ to an energy inequality in terms of $(-\Delta )^s$, which can be "localized" using the argument of Caffarelli and Silvestre \cite{CS07}. For a function $f : \mathbb{R}^d \to \mathbb{R}$, consider the extension function  $f^\star : \mathbb{R}^d \times [0, \infty) \to \mathbb{R}$ that satisfies 
\begin{equation}\label{a01}
\begin{cases}
\text{div}\big(z^{1-2s}\nabla f^\star(x,z)\big) = 0,\\
f^\star(x, 0) = f(x). 
\end{cases}
\end{equation}
Then it is well-known by now \cite{CS07} that in the weak sense 
\[
(-\Delta)^sf = \lim\limits_{z \to 0}\big(- z^{1-2s}\partial_z f^\star\big). 
\]
Furthermore, the extension $f^\star$ can be determined by using Poisson formula
\[
f^\star(x, z) =  \int_{\mathbb{R}^d}P(x - y, z)f(y)dy,
\]
where $P$ is the Poisson kernel for equation \eqref{a01} and is defined by
\[
P(x, z) = c_{d,s}\frac{z^{2s}}{\left(|x|^2 + |z|^{2}\right)^{\frac{d}{2}+s }} = \frac{1}{z^d}P\big(\frac{x}{z}, 1\big),
\]
where $c_{d,s}$ is a normalizing constant such that $\|P(\cdot, 1)\|_{L^1(\mathbb{R}^d)} = 1$. In this paper, we will denote the extension $(u(\cdot, t))^\star$ of $u(\cdot, t)$ by 
\begin{align*}
u^\star(x,z, t) = (u(\cdot, t))^\star(x, z).
\end{align*}
As previously mentioned, the starting point of the proof of H\"older regularity by De Giorgi's method is an energy-type inequality for truncations of the weak solution $u$. The following lemma, of algebraic nature, relates the energy associated to $\mathcal L_t$ to the one associated to the fractional laplacian.
\begin{lemma}\label{crucialLemma}
	Let $h$ be  a positive function in $\mathbb{R}^d$ and $f\in H^{s}(\mathbb{R}^d)$ for $s \in (0,1)$. Then the following estimate holds 
	\begin{align}
	&\int_{\mathbb{R}^d} h^2 f^+ \mathcal{L}_t f dx+ \int_{\mathbb{R}^d}\int_{\mathbb{R}^d} f^+(x) f^+(y)\frac{(h(x)-h(y))^2}{|x-y|^{d+2s}} dxdy\\&\quad \sim  	\int_{\mathbb{R}^d} h^2 f^+ (-\Delta)^{s} f dx+ C \int_{\mathbb{R}^d}\int_{\mathbb{R}^d} f^+(x) f^+(y)\frac{(h(x)-h(y))^2}{|x-y|^{d+2s}},\nonumber
	\end{align}
	where $f^+$ is the positive part of $f$. 
\end{lemma}
Lemma \ref{crucialLemma} ensures that one can always compare the quadratic forms $\int_{\mathbb{R}^d} h^2 f^+ \mathcal{L}_t f dx$ and $$\int_{\mathbb{R}^d} h^2 f^+ (-\Delta)^{s} f dx$$ by the Gagliardo semi-norms, which are under control by Sobolev inequalities. The $s-$harmonic extension described above then allows to control  the term $\int_{\mathbb{R}^d} h^2 f^+ \mathcal L_t f dx$ in a local fashion as in \cite{CaVa}. 

\subsection*{Main results}
We now describe the precise assumptions on the kernel and our main results. We introduce the space $\mathbb{X}^s$ which is defined by
\[
\mathbb{X}^s =
\begin{cases}
BMO(\R^d) & \text{if $s = 1/2$}, \\
\\
C^{1-2s}(\R^d) & \text{if $0 < s < 1/2$}.
\end{cases}
\] 
Moreover, we always assume that $\mathbf{B}$ is divergence free almost-everywhere in time $t \in \R$.  The kernel $\mathbf K_t$ satisfies

$$
\mathbf{K}_t:  \mathbb{R}^d \times \mathbb{R}^d \times \mathbb{R} \to [0, \infty), ~~ (x, y, t) \mapsto \mathbf{K}_t(x, y)
$$ 
is assumed to satisfy the following conditions for a given $s \in (0, 1/2]$:
\begin{itemize}
	\item[(i)] $\mathbf{K}_t$ is a measurable function;
	\item[(ii)] for every $t \in \mathbb{R}$, $\mathbf{K}_t$ is symmetric which means that
	\[
	\mathbf{K}_t(x, y) = \mathbf{K}_t(y, x), \quad \text{for a. e.} ~ (x, y, t) \in \mathbb{R}^d \times \mathbb{R}^d \times \mathbb{R};
	\]
	\item[(iii)] there exists a constant $\Lambda > 1$ such that
	\[
	\Lambda^{-1} \leq |x - y|^{d + 2s} \mathbf{K}_t(x, y) \leq \Lambda, ~~ \text{for a. e.} ~ (x, y) \in \mathbb{R}^d \times \mathbb{R}^d \times \mathbb{R};
	\]
		\item[(iv)] for a. e. $x \in \mathbb{R}^d$
		\[
			\Big|\int_{\mathbf{S}^{d-1}}\theta\mathbf{K}_t(x,x+\rho \theta)d\mathcal{H}^{d-1}(\theta)\Big| \leq \Lambda\rho^{-d}.
		\]
\end{itemize} 
In the following we drop the dependence in the time variable  $t$ in $\mathcal L_t$. Such operators have been considered in many works and clearly generalize the case of the fractional laplacian. See in particular \cite{CCV} for an equation related to ours. 

\medskip
Our first theorem is an H\"older estimate for solutions of the equation \eqref{eq1}.
\begin{theorem}\label{CaVa}Assume that the divergence free  vector field $\mathbf{B}$ satisfies the conditions 
	\begin{align*}
	\sup\limits_{(x, t) \in \mathbb{R}^{d+1}}\left|\fint_{x + B_1} \mathbf{B}(y, t)dy \right| \leq M_1 \quad \text{and} \quad  \|\mathbf{B}\|_{L^\infty(\mathbb{R}; \mathbb{X}^s)} \leq M_2,
	\end{align*}
	for some $M_1, M_2 \geq 1$ and $g$ belongs to the class 
	$$
	L^{1}(\mathbb{R}^{d+1})\cap L^{q}(\mathbb{R}^{d+1}),
	$$ 
	for some $q > (d + 1)/2s$. Le $u$ be a  weak solution to \eqref{eq1} with
	$$
	u \in L^\infty(\mathbb{R};L^2(\mathbb{R}^d))\cap L^2(\mathbb{R};H^s(\mathbb{R}^d)).
	$$ 
	Then there exists $\delta_0$ depending only on $d, s, \Lambda, q$ and $M_2$ such that the following holds true
	\begin{itemize}
		\item  Global in time:
		\begin{equation}\label{z17}
		\sup\limits_{(x, t) \in \R^{d+1}}\|u\|_{C^\alpha(Q_{\frac{\delta_0}{M_1}}(x,t))} + \|u\|_{L^\infty(\R^{d+1})} \leq C(M_2) \left(\|g\|_{L^1(\R^{d+1})} + \|g\|_{L^q(\R^{d+1})}\right),
		\end{equation}
		\item  Local in time:
		\begin{align}\label{z17'}
		\sup\limits_{x\in \R^{d}, \, t > 1}\|u\|_{C^\alpha(Q_{\frac{\delta_0}{M_1}}(x,t))} + \|u\|_{L^\infty(\R^{d}\times (1,\infty))} \leq C(M_2) \left(\|u(0)\|_{L^2(\R^{d})} \right. \\
		\left. + \|g\|_{L^1(\R^{d+1}_+)} + \|g\|_{L^q(\R^{d+1}_+)}\right). \nonumber
		\end{align}
	\end{itemize}
	for some $\alpha \in \big(0, 1 - (d+1)/2sq\big)$.
\end{theorem}

In the following we use the notations
\begin{equation*}
X_i = (t_i, x_i), \quad \|X_1-X_2\|_s = |x_1 - x_2| + |t_1 - t_2| ^{\frac{1}{2s}},
\end{equation*}
\begin{equation*}
\mathcal{R}^1_\lambda := \left\{X_1 = (x_1, t_1), X_2 = (x_2, t_2) \in \R^{d+1}_+ \, : \, 
\begin{cases}
t_1, t_2 > \lambda, \\
\\
\|X_1 - X_2\|_{s} \leq  \frac{C(M_2,M_1)\lambda}{|\log(|\lambda)|+1}
\end{cases}
\right\},
\end{equation*}
\begin{equation*}
\mathcal{R}^2_\lambda := \left\{X_1 = (x_1, t_1), X_2 = (x_2, t_2) \in \R^{d+1}_+ \, : \, \begin{cases}
t_1, t_2 > \lambda^{2s}, \\
\\
\|X_1 - X_2\|_{s} \leq  C(M_3)\min\{\lambda,\lambda^{2s}\}
\end{cases}	 \right\}.
\end{equation*}
From Theorem \ref{CaVa} we derive the following corollary.

\begin{corollary} \label{XX1}Assume $g=0$ in $\mathbb{R}^d\times (0,\infty)$.  Then, 
	\begin{equation}\label{XX2}
	||u(t)||_{L^\infty(\mathbb{R}^d)}\leq C t^{-\frac{d}{4s}}||u(0)||_{L^2(\mathbb{R}^d)}.
	\end{equation}
	Moreover, 
	\begin{itemize}
		\item If $s=\frac{1}{2}$ and $\mathbf{B}$ satisfies 
		$$ 
		\sup\limits_{(x,t)\in \mathbb{R}^{d+1}_+}\left|\fint_{x + B_1} \mathbf{B}(y, t)dy \right| \leq M_1, \quad \|\mathbf{B}\|_{L^\infty(\mathbb{R}_+;BMO (\mathbb{R}^d))} \leq M_2, \quad (M_1, M_2 \geq 1)
		$$ 
		then, for any  $\lambda\in (0,\frac{1}{2})$, we have
		\begin{equation}\label{XX3}
		\sup_{X_1, X_2 \in \mathcal{R}^1_{\lambda}}\frac{|u(X_1) - u(X_2)|}{||X_1-X_2||_{\frac{1}{2}}^\alpha}\leq C(M_2) \lambda^{-d/2-\alpha}||u(0)||_{L^2(\mathbb{R}^d)}.
		\end{equation}
		\item If $0<s<\frac{1}{2}$ and $\mathbf{B}$ satisfies 
		$$
		\|\mathbf{B}\|_{L^\infty(\mathbb{R}_+;L^\infty\cap C^{1-2s} (\mathbb{R}^d))} \leq M_3, \quad (M_3 \geq 1)
		$$  
		then, for any $\lambda\in (0,\frac{1}{2})$, we have
		\begin{equation}\label{XX4}
		\sup_{X_1, X_2 \in \mathcal{R}^2_{\lambda}}\frac{|u(X_1) - u(X_2)|}{||X_1-X_2||_s^\alpha}\leq C(M_3) \lambda^{-d/2-\alpha}||u(0)||_{L^2(\mathbb{R}^d)}.
		\end{equation}
	\end{itemize}
\end{corollary}

We now draw some conclusion in the case of dissipative SQG, namely

\begin{align}\label{SQG}
&\partial_t u +{\bf B}\cdot  \nabla u+\mathcal L_t u =g~~\text{in}~~\mathbb{R}\times\mathbb{R}^d,
\end{align}
where $s=\frac12$ and ${\bf B}= \nabla^\perp \mathcal L_t^{-1}u$. Notice that the latter is of order zero and is a Calderon-Zygmund operator. 

\begin{corollary}
Let $u$ be a weak solution of \eqref{SQG}. If $g\in \left(L^1\cap L^{q}\right)(\mathbb{R}^{d+1})$ for $q>d+1$, then  $u$ is H\"older continuous in space and time.
\end{corollary}
\begin{proof}
Proposition \ref{GLinf} gives that $u(t) \in L^\infty_t(\mathbb{R},L^\infty_x\cap L^2_x(\mathbb R^{d}))$; hence by Calderon-Zygmund estimate one has ${\bf B} \in L^\infty_t (BMO_x\cap L^2_x)$. Applying Theorem \ref{CaVa} gives actually that $u$ is $C^\alpha$. 
\end{proof}

\medskip

Throughout this work we denote  
\[
	B_r = [-r, r]^d, ~ Q_{r,\tau} = B_r \times [-\tau, 0], ~ Q_r = Q_{r,r},  
\]
\[
	B_{r,\rho}^{\star}=B_r\times (0,\rho), ~ B^\star_r = B^\star_{r,r}, ~ Q^\star_r = B_r^{\star} \times [-r, 0].
\]

\setcounter{equation}{0}

\section{From $L^2$ to $L^\infty$}

\qquad This section is devoted to the first step of De Giorgi's method in proving the H\"older continuity of solutions. We prove that weak solution is actually bounded.  Our result is the following:

\begin{proposition}\label{GLinf} Let $q > \frac{d+1}{2s}$ and assume that $g$ belongs to the class $$L^{\frac{2(d+1)}{d+1+2s}}(\mathbb{R}^{d+1})\cap L^{q}(\mathbb{R}^{d+1}).$$ Then there exists a unique solution of the equation \eqref{eq1} satisfying following estimates
	\begin{equation}\label{z8}
		\sup_{t\in \mathbb{R}}\int_{\mathbb{R}^d}u(t)^2dx+\int_{\mathbb{R}^{d+1}}\big|(-\Delta)^{\frac{s}{2}} u\big|^2 dxdt\leq C\|g\|_{L^{\frac{2(d+1)}{d+1+2s}}(\mathbb{R}^{d+1})}^2,
	\end{equation}
	\begin{equation}\label{z9}
		\|u\|_{L^{\infty}(\mathbb{R}^{d+1})} \leq C\|g\|_{L^{\frac{2(d+1)}{d+1+2s}}(\mathbb{R}^{d+1})}^{\frac{2(q-d-2s)}{2(q-d-2s)+qd}}||g||_{L^{q}(\mathbb{R}^{d+1})}^{\frac{qd}{2(q-d-2s)+qd}}.
	\end{equation}
\end{proposition}

\begin{proof} By a standard argument, we get the existence and uniqueness of a solution $u$ for the equation \eqref{eq1} satisfying the estimate \eqref{z8}.

\vspace{0.2cm}\medskip 
\quad In order to prove \eqref{z9} we only need to show that 
	\begin{equation}\label{z10}
		\|u^{+}\|_{L^{\infty}([1,\infty)\times\mathbb{R}^{d})}\leq C,
	\end{equation}
	provided that
	\[
		\|g\|_{L^{\frac{2(d+1)}{d+1+2s}}(\mathbb{R}^{d+1})} + \|g\|_{L^{q}(\mathbb{R}^{d+1})}\leq 1.
	\] 
We shall use the De Giorgi's iteration. For the constant $\lambda > 1$ that will be chosen later, we consider the levels 
\[
	C_k = \lambda(1-2^{-k})
\]
and the truncation functions $u_k=(u - C_k)_+$. By using $u_k$ as test function of  \eqref{eq1}, we obtain: 
\begin{equation*}
	U_k \leq C\int_{\mathbb{R}^d} u_k(\tau)^2dx + C\int_{T_{k-1}}^{\infty}\int_{\mathbb{R}^d}u_{k}|g|dxdt,
\end{equation*}
for any $\tau \in [T_{k-1}, T_k]$, where $T_k = 1 - 2^{-k}$ and $U_k$ is defined by
\begin{equation*}
	U_k=\sup_{t\geq T_k}\int_{\mathbb{R}^d}u_k(t)^2dx + \int_{T_k}^{\infty}\int_{\mathbb{R}^{d}}\big|(-\Delta)^{s/2} u_k\big|^2 dxdt.
\end{equation*}	
Taking the mean value in $\tau$ on $[T_{k-1},T_k]$ we find, 
\begin{equation*}
U_k\leq 2^{k}C\|u_k\|_{L^{2}([T_{k-1},\infty)\times\mathbb{R}^d)}^2 + C\|u_k\|_{L^{q'}([T_{k-1},\infty)\times\mathbb{R}^d)},
\end{equation*}
where $$q'=\frac{q}{q-1}<\frac{d+2s}{d+1-2s}.$$ As in \cite{CaVa} we can control the righe-hand side of this inequality by $U_{k-1}$. Indeed, by using Sobolev's inequality  we get
\begin{equation*}
	\|u_{k-1}\|^2_{L^{\frac{2(d+1)}{d+1-2s}}([T_{k-1},\infty)\times\mathbb{R}^d)} \leq CU_{k-1}.
\end{equation*}
On the other hand, it is noted that
\[
	\mathbf{1}_{\{u_k > 0\}} \leq \left(\frac{2^k}{\lambda}u_{k-1}\right)^{\frac{4s}{d+1-2s}} \quad \text{and} \quad  \mathbf{1}_{\{u_k > 0\}} \leq \left(\frac{2^k}{\lambda}u_{k-1}\right)^{-q' + \frac{2(d+1)}{d+1-2s}}.
\]
These imply that
\begin{align*}
	\|u_k\|_{L^{2}([T_{k-1},\infty)\times\mathbb{R}^d)}^2 \leq\left( \frac{2^{k}}{\lambda}\right)^{\frac{4s}{d+1-2s}} \|u_{k-1}\|_{L^{\frac{2(d+1)}{d+1-2s}}([T_{k-1},\infty)\times\mathbb{R}^d)}^{\frac{2(d+1)}{d+1-2s}},
\end{align*}
and
\begin{align*}
	\|u_k\|_{L^{q'}([T_{k-1},\infty)\times\mathbb{R}^d)}\leq \left(\frac{2^k}{\lambda}\right)^{\frac{2(d+1)}{(d+1-2s)q'}-1} \|u_{k-1}\|_{L^{\frac{2(d+1)}{d+1-2s}}([T_{k-1},\infty)\times\mathbb{R}^d)}^{\frac{2(d+1)}{(d+1-2s)q'}}.
\end{align*} 
We combine above estimates and obtain
\begin{equation*}
U_k\leq C  \frac{2^{k(d+1+2s)/(d+1-2s)}}{\lambda^{4s/(d+1-2s)}} U_{k-1}^{\frac{d+1}{d+1-2s}}+C \left(\frac{2^k}{\lambda}\right)^{\frac{2(d+1)}{(d+1-2s)q'}-1}U_{k-1}^{\frac{d+1}{(d+1-2s)q'}}, ~~\forall~k\geq 1. 
\end{equation*}
Thanks to \eqref{z8},  one has $U_0\leq C$. Since $$\frac{d+1}{d+1-2s} > 1 ~ \text{and} ~ \frac{d+1}{(d+1-2s)q'}>1,$$ so  for $\lambda>1$ large enough, $U_k$ converges to $0$. This means $u \leq \lambda$ for almost everywhere in $[1,\infty)\times\mathbb{R}^d$ and we find \eqref{z10}. The proof is complete. 
\end{proof}
\begin{remark}\label{ZA} From the  proof of the inequality \eqref{z10}, we find that if 
	\begin{equation}
	\|u(0)\|_{L^{2}(\mathbb{R}^{d})}+	\|g\|_{L^{\frac{2(d+1)}{d+1+2s}}(\mathbb{R}^{d+1}_+)} + \|g\|_{L^{q}(\mathbb{R}^{d+1}_+)}\leq 1, 
	\end{equation}
	then there holds
		\begin{equation}\label{z10'}
	\|u\|_{L^{\infty}([1,\infty)\times\mathbb{R}^{d})}\leq C.
	\end{equation}
\end{remark}

\medskip
\medskip 

\section{Oscillation lemma}

\qquad This section is concerned with reducing the oscillation of the solution, one of the most difficult task in proving the H\"older continuity of bounded solutions. In the critical case, Caffarelli and Vasseur used the harmonic extension $u^\star$ of solution $u$ in $\R^d \times \R^+$ . Namely, they expressed the fractional Laplacian $(-\Delta)^\frac{1}{2}u$ as the normal derivative of the harmonic extension  $u^\star$ on the boundary $\mathbb{R}^d$. Then thanks to the good properties of harmonic functions, they obtained the following diminishing oscillation result: if $ u^\star \leq 2$ in a box centered at the origin then $u^\star$ satisfies  
\[
	u^* < 2 - \lambda^\star
\]
in a smaller box, for some $\lambda^\star > 0$ which depends on the BMO-norm of the vector field $\mathbf{B}$. From this diminishing oscillation result, the H\"older continuity of $u^\star$ is proved by construction a suitable sequence of functions and by using the natural scaling invariance.

\vspace{0.2cm}
In the supercritical case, there is a change in the scaling invariance. Following the idea of Caffarelli and Vasseur with some modifications, Constantin and Wu  \cite{Cons-Wu} also derived a similar result.

\vspace{0.2cm} 
It seems that the approach of Caffarelli and Vasseur can not be used directly for the general kernel case because we have no information relating our operator in terms of some extension. In our case we then use an algebraic result to compare our integral operator with the standard fractional Laplacian. This allows us to continue using $s$-harmonic extension to obtain the following disminishing oscillation result:

\begin{proposition}\label{osclem1} Assume that the vector field $\mathbf{B}$ and the function $g$ satisfy following conditions 
	\begin{align*}
	\|\mathbf{B}\|_{L^\infty(-4,0;L^{d/s}(B_2))}\leq M_0, ~~ \displaystyle \|g\|_{L^{q}(Q_4)}^2 \leq \hat{\varepsilon}_0,
	\end{align*}
	where $\hat{\varepsilon}_0 < \varepsilon_0$ ($\varepsilon_0$ is the small constant defined in \ref{lem1}). Then there are positive constants $\lambda^\star$, $\varepsilon_2$ depending only on $s, d, M_0$  such that, for every solution $u$ of \eqref{eq1}, the following holds true:  
	
	\vspace{0.2cm}
	\qquad If $u$ satisfies 
	\begin{gather*}
	u^* \leq 2 ~~\text{in}~~Q^\star_1,  ~~~ \int_{B^c_1} \frac{(u(x, t)-2)_+}{|x|^{d+2s}}dx \leq  \varepsilon_2 \quad \forall \, t \in \R, \\
	\mathcal{L}_\omega^{d+2}\left(\left\{(x,z,t)\in Q_1^\star: u^*(x,z,t)\leq 0\right\}\right)\geq \frac{|B_1|}{4(1-s)},
	\end{gather*}
	then we have
	\begin{equation}\label{z7}
	u^* \leq 2-\lambda^\star~~\text{in}~~Q^\star_{1/16}.
	\end{equation}
\end{proposition}

\vspace{0.2cm}
The proof of this proposition is related to three propositions: a local energy inequality,  a result on the diminishing oscillation for $u^\star$  under  conditions the local $L^2$-norm of $u$ and  $u^\star$ are small. Third proposition shows the sufficient conditions to obtain the smallness of the local $L^2$-norm of $u$ and  $u^\star$.  

\medskip
\begin{proposition}\label{prop2}  Let $t_1, t_2$ be real numbers with $t_1 < t_2$. Assume that the vector field $\mathbf{B}$ with divergence free satisfies the condition 
\begin{align*}
	\|\mathbf{B}\|_{L^\infty(t_1,t_2;L^{d/s}(B_2))}^2\leq M_0,
\end{align*}
and let $\phi, \psi$ be cut-off functions  in $\mathbb{R}^d$ and $\mathbb{R}$ respectively such that 
\[
	\mathbf{1}_{B_1}\leq \phi \leq \mathbf{1}_{B_2} \quad \text{and} \quad \mathbf{1}_{(-1,1)}\leq \psi \leq \mathbf{1}_{(-2,2)}.
\]
If $u$ be a solution to \eqref{eq1} with
\[
	u\in L^\infty(t_1,t_2;L^2(\mathbb{R}^d)), \quad (-\Delta)^{s/2}u\in L^2((t_1,t_2)\times\mathbb{R}^d),
\]
then there exists a positive constant $C = C(d, \Lambda)$  such that 
\begin{align}\label{energe1}
	\frac{d}{dt}\int_{B_2}& (\phi u_+)^2dx  +	\int_{B_2^*} z^{1-2s}|\nabla(\eta u^\star_+)|^2dxdz   \notag \\ 
	& \leq C \iint_{\mathbb{R}^d\times\mathbb{R}^d} u_+(x) u_+(y)\frac{(\phi(x)-\phi(y))^2}{|x-y|^{d+2s}}dxdy \\ 
	& \quad + C\int_{B_2^*}z^{1-2s}(|\nabla \eta| u^\star_+)^2 dxdz   + CM_0 \int_{B_2}|\nabla \phi|^2u_+^2 dx + C \int_{B_2} \phi^2u_{+} |g| dx, \nonumber
\end{align}
	for every $t_1\leq t\leq t_2$. Here we denote $\eta =\phi \psi$.
\end{proposition}

\begin{proof} 
	Using $\phi^2 u_{+}$ as test function of equation \eqref{eq1}, ones has 
	\begin{align}\label{Z1}
		\frac{d}{dt}\int_{B_2}\phi^2 \frac{u_+^2}{2}dx  + \int_{B_2} \phi^2u_{+}\mathcal{L}_t udx = \int_{B_2} \mathbf{B}.\nabla\phi^2\frac{u_+^2}{2}dx  + \int_{B_2} \phi^2u_{+} gdx. 
	\end{align}
	By  using Holder's inequality it follows that
	\begin{align*}
		\left|\int_{B_2} \mathbf{B}.\nabla \phi^2\frac{u_+^2}{2}dx\right| & \leq \varepsilon \|\phi u_+\|_{L^{\frac{2d}{d-2s}}(B_2)}^2 + \frac{1}{\varepsilon}\|(\mathbf{B}.\nabla \phi)u_+\|_{L^{\frac{2d}{d+2s}}(B_2)}^2 \\
		&\leq \varepsilon ||\phi u_+||_{L^{\frac{2d}{d-2s}}(B_2)}^2 + \frac{M_0}{\varepsilon}\int_{B_2}|\nabla \phi|^2u_+^2 dx. \nonumber
	\end{align*}
	The first term in the RHS can be estimated by using the trace theorem and the Sobolev embedding as follows
	\begin{align*}
		\|\phi u_+\|_{L^{\frac{2d}{d-2s}}(B_2)}^2 & \leq C\|\mathbf{1}_{B_2} \phi u_+\|_{H^{s}(\mathbb{R}^d)}^2 \\
		& = C\int_0^\infty\int_{\mathbb{R}^d}z^{1-2s}|\nabla (\mathbf{1}_{B_2} \phi u_+)^\star|^2dxdz	\\
		& \leq C\int_0^\infty\int_{\mathbb{R}^d}z^{1-2s}|\nabla[\mathbf{1}_{B_2^*} \eta u^\star_+ |^2dxdz \\
		& = C\int_{B_2^*}z^{1-2s}\left|\nabla[\eta u^\star_+]\right|^2dxdz.	
	\end{align*}
	Hence it implies
	\begin{align}\label{Z2}
		\left|\int_{B_2} \mathbf{B}.\nabla \phi^2\frac{u_+^2}{2}dx\right| \leq \varepsilon\int_{B_2^*}z^{1-2s}\left|\nabla[\eta u^\star_+]\right|^2dxdz+ C_\varepsilon M_0\int_{B_2}|\nabla \phi|^2(u_+)^2 dx.
	\end{align}
		
	\noindent 
	Next, thanks to  Lemma \ref{crucialLemma}, it follows that
	\begin{align*}
		\int_{B_2} \phi^2 u_{+}\mathcal{L}_t udx \geq	C_0\int_{B_2} \phi^2u_{+}(-\Delta)^{s/2}u dx - C_1 \iint_{\mathbb{R}^d\times\mathbb{R}^d}  u_+(x)u_+(y)\frac{(\phi(x)-\phi(y))^2}{|x-y|^{d+2s}}dxdy. 	\end{align*}
	On the other hand, we have
	\begin{align*}
		0 & = \int_{B_2^*}\eta^2 u^\star_+\text{div}(z^{1-2s}\nabla u^\star)dxdz \\ 
		   & = -\int_{B_2^*} z^{1-2s}\nabla(\eta^2 u^\star_+)\nabla u^\star dxdz + \int_{B_2}\phi^2u_+(-\Delta)^{s/2}u dx \\
		   & = -\int_{B_2^*} z^{1-2s}|\nabla(\eta u^\star_+)|^2dxdz + \int_{B_2^*}z^{1-2s}|\nabla\eta|^2 (u^\star_+)^2dxdz + \int_{B_2}\phi^2u_+(-\Delta)^{s/2}u dx.
	\end{align*}
	Hence we imply
	\begin{align}\label{Z3}
		\int_{B_2} \eta^2[u]_{+}\mathcal{L}_t u \geq & C_0\int_{B_2^*} z^{1-2s}|\nabla(\eta u^\star_+)|^2dxdz - C_0 \int_{B_2^*}z^{1-2s}|\nabla\eta|^2 u^\star_+dxdz \\
		& - C_1 \iint_{\mathbb{R}^d\times\mathbb{R}^d}  u_+(x) u_+(y)\frac{(\phi(x)-\phi(y))^2}{|x-y|^{d+2s}}dxdy. \nonumber
	\end{align}
	By choosing $\varepsilon$ small enough and by combining \eqref{Z1}, \eqref{Z2}, \eqref{Z3} yield  \eqref{energe1}.
\end{proof}

\begin{remark}
	\begin{itemize}
		\item[$\diamond$] This proposition allows us to  control the $L^2_tL^2_{x,z}$-norm of the gradient of $u^\star$. Howerver, it gives a control in $L_t^\infty L^2_{x,z}$-norm of $u^\star$ only on the trace $z = 0$. In oder to obtain the regularity results, it remains to control this norm of $u^\star$ in $z > 0$.
		\item[$\diamond$] In order to control the term 
		\[
		 \iint_{\mathbb{R}^d\times\mathbb{R}^d}  u_+(x) u_+(y)\frac{(\phi(x)-\phi(y))^2}{|x-y|^{d+2s}}dxdy
		\]
		we need to add the condition $\displaystyle \int_{B_4^c} \frac{(u(x, t)-2)_+}{|x|^{d+2s}} dx \leq 2, \, \forall t \in \R$ in the sequel estimates.
	\end{itemize}
\end{remark}

\begin{proposition}\label{lem1} Let $q>(d+1)/2s$. Assume that the vector field $\mathbf{B}$ with divergence free satisfies 
	\begin{align}
		\|\mathbf{B}\|_{L^\infty(-4,0;L^{d/s}(B_2))}\leq M_0.
	\end{align}
	There exist constants $\varepsilon_0 = \varepsilon_0(s, d, \Lambda, M_0) > 0$ and  $\lambda = \lambda(s, d) \in (0,\frac{1}{2})$ such that for any solution $u$ to equation \eqref{eq1}, the following property holds true:\\
	If 
	\begin{align}\label{z01}
		\int_{B_4^c} \frac{(u(x, t)-2)_+}{|x|^{d+2s}} dx \leq 2, \, \forall t \in \R,
	\end{align}
	\begin{align}\label{z02}
		u^\star \leq 2~~\text{in}~~ Q_4^\star,
	\end{align}
	and 
	\begin{align}\label{z4}
		\int_{Q_4^{\star}}z^{1-2s} (u^\star_+)^2 dxdzds + \int_{Q_4} u_+^2 dxds   +||g||_{L^{q}(Q_4)}^2\leq \varepsilon_0,
	\end{align}
	then
	\begin{align}
		u_+ \leq 2-\lambda ~~ \text{on} ~~ Q_1.
	\end{align}
\end{proposition}
\begin{proof} We follow the proof of \cite[Lemma 6]{CaVa}.  For  $k = 0, 1, 2, ...$, we set $$C_k=2-\lambda(1+2^{-k}),$$ 
	 \[
		u_k=(u-C_k)_+, ~~~ u^*_k= (u^\star-C_k)_+.
	\]
By using the  energy inequality \eqref{energe1} with the test function $\phi_k\psi u_k$, where 
	\begin{itemize}
		\item[$\diamond$] $\phi_k$ is a cut off function in $x$ such that
		\[
			\mathbf{1}_{B_{1+2^{-k-1/2}}}\leq \phi_k \leq \mathbf{1}_{B_{1+2^{-k}}} ~~\text{and} ~~  |\nabla\phi_k|\leq C 2^{k},
		\]
		\item[$\diamond$] $\psi$ is a fixed cut-off function in $z$ only with $\mathbf{1}_{(-1,1)}\leq \psi \leq \mathbf{1}_{(-2,2)}$,
	\end{itemize}
it follows that for any $t_1\in [-4,-2]$ we have
\begin{align*}
A_k & \leq \sup_{t_2\in [-1-2^{-k},0]}	\left(\int_{B_2}\phi_k^2 u_k^2(t_2)dx+	\int_{t_1}^{t_2}\int_{B^{\star}_2}z^{1-2s} |\nabla \left(\eta_k u^*_k\right)|^2dxdzdt\right)\\ 
	& \leq C\int_{-4}^{0} \iint_{\mathbb{R}^d\times\mathbb{R}^d}   u_k(x) u_k(y)\frac{(\phi_k(x)-\phi_k(y))^2}{|x-y|^{d+2s}}dxdy dt\\
	& ~~~~ + \int_{B_2}\phi_k^2 u_k^2(t_1)dx + C\int_{-4}^{0}\int_{B^{\star}_2} z^{1-2s}|\nabla\eta_k|^2(u^*_k)^2dxdzdt\\
	& ~~~~ + CM_0\int_{-4}^{0}\int_{B_2}|\nabla \phi_k|^2u_k^2 dxds + \int_{-4}^{0}\int_{B_2} \phi^2_ku_k |g| dxdt.
\end{align*}
Here $A_k$ is defined by
\begin{align*}
	A_k=\int_{-1-2^{-k}}^{0}\int_{0}^{\delta^k}\int_{\mathbb{R}^d}z^{1-2s}|\nabla(\eta_k u^*_k)|^2dxdzdt+\sup_{t\in [-1-2^{-k},0]}\int_{\mathbb{R}^d}(\phi_k u_k)^2dx,
\end{align*}
for some constant $\delta > 0$ that will be chosen later. 

\vspace{0.2cm}
\quad Thanks to the conditions \eqref{z4} and the facts that  $|\nabla \phi_k|\leq C2^{k}$ and
\[
	\int_{-4}^{0}\int_{B_2} \phi^2_ku_k |g| dxdt \leq \frac{1}{2}\big(\int_{Q_4}u_k^2dxdt + \|g\|^2_{L^2(Q_4)}\big),
\]
we obtain
\begin{align*}
	A_k\leq  C(1+M_0)2^{2k}\varepsilon_0 + C\int_{-4}^{0}\int_{\mathbb{R}^d}\int_{\mathbb{R}^d} u_k(x) u_k(y)\frac{(\phi_k(x)-\phi_k(y))^2}{|x-y|^{d+2s}}dxdy dt.
\end{align*}
On the other hand we note that
\begin{align*}
	\int_{-4}^0\int_{B_4}\int_{B_4} u_k(x) u_k(y)  \frac{(\phi_k(x)-\phi_k(y))^2}{|x-y|^{d+2s}}dxdydt & \leq C\int_{-4}^0\int_{B_4}\int_{B_4}\frac{u_k(x) u_k(y)}{|x - y|^{d-2+2s}}dxdydt \\
 	& \leq C \|u_k\|_{L^2(Q_4)}^2 \leq C\varepsilon_0,
\end{align*}
and
\begin{align*}
 	\int_{-4}^0\int_{B_4^c}\int_{B_4}  u_k(x) u_k(y)\frac{(\phi_k(x)-\phi_k(y))^2}{|x-y|^{d+2s}}dxdydt 
&	\leq \int_{-4}^0\int_{B_4^c}\int_{B_2} \frac{u_k(x)u_k(y)}{|x - y|^{d+2s}}dxdydt \\
	& \leq C\int_{-4}^0\int_{B_4^c}\frac{u_k(y)}{|y|^{d+2s}}dy\int_{B_2}u_k(x)dxdt 
	\\&\leq C \|u_k\|_{L^2(Q_4)} \leq C \sqrt{\varepsilon_0}.
\end{align*}
Here we used \eqref{z01} and \eqref{z4} in two last inequalities.
Hence it follows that
\begin{equation}\label{z03}
	A_k \leq C\Big[(1 + M_0)2^{2k}\varepsilon_0 + \sqrt{\varepsilon_0}\Big].
\end{equation}

\vspace{0.2cm}
 In the following we shall prove that there exist $0<\delta<1$ and $M>1$ such that  the following estimates hold true:
\begin{equation}
	A_k	\leq M^{-k}\label{z2},
\end{equation}
\begin{equation}
	\eta_ku^*_k = 0~~\text{for}~~\delta^k\leq z\leq \min\{2,\delta^{k-1}\},\label{z3}
\end{equation}
for every $k \geq 0$,

\vspace{0.2cm}\noindent
{\bf Step 1.}  (Existence of $\delta$ and $M$)

\vspace{0.2cm}
\quad	Let  $\xi_1$ and $\xi_2$ be the solutions of respectively following problems
	\[
		\begin{cases}
			\mathrm{div}(z^{1-2s}\xi_1) = 0 & \text{in}~~ B_4^*, \\
			\xi_1 = 2 & \text{in} ~~ \partial B_4^* \setminus \{z = 0 \}, \\
			\xi_1 = 0 & \text{in} ~~ \partial B_4^* \cap \{z = 0\},
		\end{cases}
	\]	
	and
	\[
		\begin{cases}
			\mathrm{div}(z^{1-2s}\xi_2) = 0 & \text{in}~~ [0, \infty) \times [0, 1], \\
			\xi_2(0, z) = 2 & 0 \leq z \leq 1, \\
			\xi_2(x, 0) = \xi_2(x, 1) = 0 & 0 < x< \infty.
		\end{cases}
	\] 
	From the maximum principle we can choose a constant $\lambda \in (0,1/2)$ such that 
	\begin{align}\label{z5}
		\xi_1(x,z)\leq 2-4\lambda ~~\text{in}~~B_2^{\star}.
	\end{align}	 
	Moreover, there are also constants $c_1 > 0$ and $\sigma_0$ satisying
	\begin{align}\label{z1}
		|\xi_2(x,z)| \leq c_1e^{-\sigma_0x }~~\text{in}~~[0,\infty)\times[0,1].
	\end{align}

\quad Let $c_2$ and $\theta_1, \theta_2 > 1$ be constants that will be chosen in later.   It is not difficult to check that (see more details in \cite[Lemma 7]{CaVa}) there are constants $\delta \in (0, 1)$ and $M > 1$ satisfying
\begin{align}\label{za1}
	&2dc_1\exp\Big({-\frac{2^{-k}}{4(\sqrt{2}+1)\delta^k}}\Big)\leq \lambda 2^{-k-2},
\end{align}
\begin{equation}\label{za2}
	 M^{-k/2}\delta^{-\frac{d(k+1)}{2}} ||P(\cdot, 1)||_{L^2(\mathbb{R}^d)}\leq \lambda 2^{-k-2},
\end{equation}
and
\begin{equation}\label{za3}
	M^{-k}\geq c_2^k\big(M^{-\theta_1(k-3)} + M^{-\theta_2(k-3)}\big),
\end{equation}
for all $k\geq 12d$. 

\vspace{0.2cm}\noindent 
{\bf Step 2.} (Initial step)

\vspace{0.2cm}

\qquad  From \eqref{z03} we can choose $\varepsilon_0$ sufficiently small so that \eqref{z2} is verified for all $0\leq k\leq 12d$. Next, we shall prove the equality \eqref{z3} with $k=0$. Similarly as in proof of \cite[Lemma 6]{CaVa} we can use the maximum principle to obtain
\[
	u^* \leq (\mathbf{1}_{B_4}u_+) * P(\cdot, z) + \xi_1 ~~ \text{in} ~~ B_4^* \times \mathbb{R}^+.
\]
On the other hand, it follows from Young's inequality and \eqref{z4} that, for $t \in [-4, 0]$ and $z \geq 1$ that
\begin{align*}
	\left\|(\mathbf{1}_{B_4}u_+) * P(\cdot, z)\right\|_{L^{\infty}(B_4)} & \leq C\left\|P(\cdot, z)\right\|_{L^{2}(\mathbb{R}^d)} \left\|(\mathbf{1}_{B_4}u_+)\right\|_{L^{2}(\mathbb{R}^d)} \\
	& \leq C\|P(\cdot, 1)\|_{L^{2}(\mathbb{R}^d)}\|\mathbf{1}_{B_4}u_+\|_{L^2(B_4)} \leq C\sqrt{\varepsilon_0}.
\end{align*}
So thanks to \eqref{z5} we can choose $\varepsilon_0$ small enough so that 
\[
	u^* \leq 2 - 2\lambda ~~\text{in} ~~ B_4 \times [1, \infty) \times [-4, 0],
\]
which implies $\eta_0u^*_0$ vanishes for $1 \leq z \leq 2$.

\vspace{0.2cm}\noindent 
{\bf Step 3.} (Propagation of Properties \eqref{z2} and \eqref{z3})

\qquad Let us assume that \eqref{z2} and \eqref{z3} hold true for $k = m$. We shall prove that \eqref{z3} is also true at $k = m+1$. To do this we consider the function $\zeta_k$ as follow
\[
	\zeta_k(x,z) = \sum_{i=1}^d\hat{\xi}_i(x, z) + \big((\eta_mu_m) \star P(\cdot, z)\big)(x).
\]
Here, for $x = (x_1, x_2, ..., x_d)$, the function $\hat{\xi}_i$ is defined by
\[
	\hat{\xi}_i(x, z) = \xi_2\big(\frac{-x_i + x^+}{\delta^m}, \frac{z}{\delta^m} \big) + \xi_2\big(\frac{x_i - x^-}{\delta^m}, \frac{z}{\delta^m} \big),
\]
with $x^+ = 1 + 2^{-m - 1/2}$, $x^- = -x^+$. Then we can see that $\zeta_m$ satisfies the equation
\begin{align*}
	 	\mathrm{div}(z^{1-2s}\nabla\zeta_m) = 0  ~~ \text{in} ~~ Q^\delta_m :=  B_{1+2^{-m-1/2}} \times (0, \delta^m).
\end{align*}
It is clear that $\zeta_m$ vanishes on the side $z = \delta^m$ thanks to the induction assumptions and the definition of $\xi_2$. It is $\eta_mu_m \star P(\cdot, \cdot)$ on the side $z = 0$. Moreover, on the side $x_i = x^+$ or $x_i = x^-$ then it is bigger that $2$. Hence by maximum principle we have
\[
	u^*_m \leq \zeta_m ~~ \text{in} ~~ Q_m^\delta.
\]
On the other hand, it follows from Step 1 that for every $x \in B_{1+2^{-m-1}}$ and $z \in [0, \delta^m]$ we have
\begin{align*}
	\zeta_m(x, z) &\leq 2dc_1\exp\Big({-\frac{2^{-m}}{4(\sqrt{2}+1)\delta^m}}\Big) + \big((\eta_mu_m) \star P(\cdot, z)\big)(x)\\
	&\leq \lambda 2^{-m-2} + \big((\eta_mu_m) \star P(\cdot, z)\big)(x).
\end{align*}
Hence, for $x \in B_{1+2^{-m-1}}$ and $z \in [0, \delta^m]$, 
$$
	u^*_{m+1}  \leq \big(u^*_m - \lambda 2^{-m-1} \big)_+\leq \big((\eta_mu_m) \star P(\cdot, z) -\lambda 2^{-m-2} \big)_+.$$
This implies that
\begin{equation}\label{a1}
	\eta_{m+1}u^*_{m+1}  \leq \eta_{m+1}\big((\eta_mu_m) \star P(\cdot, z) -\lambda 2^{-m-2} \big)_+.
\end{equation}
From step 1 and the fact $A_m\leq M^{-m}$, we can use H\"older's inequality and \eqref{za2} to find that 
\[
	\big|(\eta_mu_m) \star P(\cdot, z)\big| \leq A_m^{1/2}\|P(\cdot, z)\|_{L^2(\mathbb{R}^d)} \leq \frac{M^{-m/2}}{\delta^{(m+1)d/2}}\|P(\cdot, 1)\|_{L^2(\mathbb{R}^d)}  \leq \lambda 2^{-m-2},
\]
for $\delta^{m+1}\leq z \leq \delta^m$. Combining this with \eqref{a1} yields 
\[
	\eta_{m+1}u^*_{m+1} = 0 ~~ \text{for} ~ \delta^{m+1} \leq z \leq \delta^m.
\]

\vspace{0.2cm}
\quad Let $m \geq 12d + 1$. Assume that \eqref{z2} is true for $k = m - 3, m - 2, m-1$ and \eqref{z3} is true for $k = m - 3$. In order to prove the estimate \eqref{z2} holds true at $k = m$ it is enough to show that 
\begin{align}\label{z6}
	A_m	\leq c_2^m \left(A_{m-3}^{\theta_1} + A_{m-3}^{\theta_2}\right), ~~ \theta_1, \theta_2 > 1.
\end{align}
for some $\theta_1,\theta_2>1$.
Since the functions $$\eta_{m-3}u^*_{m-3}\mathbf{1}_{\{0 < z < \delta^{m-3}\}} ~~ \text{and} ~~ (\eta_{m-3} u_{m-3})^*$$ have the same trace at $z = 0$ we imply 
\begin{align*}
&	\int_0^{\delta^{m-3}}\int_{\mathbb{R}^d}z^{1-2s}\left|\nabla\left(\eta_{m-3}u^*_{m-3}\right)\right|^2dxdz 
	\\& ~~~~~= \int_0^{\infty}\int_{\mathbb{R}^d}z^{1-2s}\left|\nabla\left[\eta_{m-3}u^*_{m-3}\mathbf{1}_{\{0 < z < \delta^{m-3}\}}\right]\right|^2dxdz \\ &~~~~~
	 \geq \int_0^{\infty}\int_{\mathbb{R}^d}z^{1-2s}\left|\nabla\left[(\eta_{m-3} u_{m-3})^*\right]\right|^2dxdz \\ & ~~~~~
	 = \int_{\mathbb{R}^d}\big|(-\Delta)^{s/2}(\phi_{m-3} u_{m-3})\big|^2dx.
\end{align*}
By integrating with respect to $t$ over $[- 1 - 2^{-m-3}, 0]$ we obtain
\begin{align*}
	&\int_{-1 - 2^{-m-3}}\int_0^{\delta^{m-3}}\int_{\mathbb{R}^d}z^{1-2s}\left|\nabla\left(\eta_{m-3}u^*_{m-3}\right)\right|^2dxdzdt \\
	&\qquad \qquad \geq \int_{-1 - 2^{-m-3}}\int_{\mathbb{R}^d}\big|(-\Delta)^{s/2}(\phi_{m-3} u_{m-3})\big|^2dxdt.
\end{align*}	
Hence, by using the Sobolev's inequality it follows from the definition of $A_k$ that 
\[
	A_{m-3} \geq C\|\phi_{m-3} u_{m-3}\|^2_{L^{\frac{2(d+1)}{d+1-2s}}(\mathcal{O}_{m-3})}, ~~ \text{with} ~~ \mathcal{O}_k =   \mathbb{R}^d \times [-1-2^{-k}, 0].
\]
On the other hand, from \eqref{a1} we have 
\[
	\eta_{m-2}u^*_{m-2} \leq \left[(\eta_{m-3}u^*_{m-3}) * P(z)\right]\eta_{m-2} ~~ \text{on}~~ \mathcal{O}_{m-3}^\star,
\]
with $ \mathcal{O}_{k}^\star = \mathbb{R}^d \times [0, \delta^{k}] \times [-1 - 2^{-k}, 0].$ Hence it follows from Young's inequality that
\[
	\left\|\eta_{m-2}u^*_{m-2} \right\|^2_{L^{\frac{2(d+1)}{d+1-2s}}(\mathcal{O}^\star_{m-2})} \leq \|P(1)\|^2_{L^1(\mathbb{R}^d)}\|\eta_{m-3} u_{m-3}\|^2_{L^{\frac{2(d+1)}{d+1-2s}}(\mathcal{O}_{m-3})}.
\]
Thus we obtain
\begin{align}\nonumber
	A_{m-3} & \geq C\left(\left\|\eta_{m-2}u^*_{m-2} \right\|^2_{L^{\frac{2(d+1)}{d+1-2s}}(\mathcal{O}^\star_{m-2})} +  \|\eta_{m-3} u_{m-3}\|^2_{L^{\frac{2(d+1)}{d+1-2s}}(\mathcal{O}_{m-3})}\right)\\
	&\geq C\left(\left\|\eta_{m-2}u^*_{m-1} \right\|^2_{L^{\frac{2(d+1)}{d+1-2s}}(\mathcal{O}^\star_{m-2})} +  \|\eta_{m-2} u_{m-1}\|^2_{L^{\frac{2(d+1)}{d+1-2s}}(\mathcal{O}_{m-2})}\right).\label{Z6}
\end{align}

\vspace{0.1cm}\noindent
Next, by the definition of $A_k$ we have
\[
	A_m  \leq \sup_{t_2\in [-1-2^{-m},0]}	\left(\int_{\mathbb{R}^d}\phi_{m}^2 u_{m}^2(t_2)dx \right.  \left.+	\int_{t_1}^{t_2}\int_0^{\delta^m}\int_{\mathbb{R}^d}z^{1-2s}|\nabla \left(\eta_{m} u^*_{m}\right)|^2dxdzdt\right),
\]
for any $t_1 \in [-1 - 2^{-m+1}, -1 - 2^{-m}]$. Then we can apply Proposition \ref{prop2} to obtain 
{
\begin{align*}
A_m &\lesssim  \int_{-1-2^{-m+1}}^{0}\int_{0}^{\delta^m}\int_{B_2} z^{1-2s}|\nabla\eta_{m}|^2(u^*_{m})^2dxdzdt\\
& \quad + M_0\int_{-1-2^{-m+1}}^{0}\int_{B_2}|\nabla \eta_{m}|^2u_{m}^2 dxds + \int_{-1-2^{-m+1}}^{0}\int_{B_2}\eta_m^2u_m|g| dxdt \\
& \quad + \int_{-1-2^{-m+1}}^{0}\int_{\mathbb{R}^d}\int_{\mathbb{R}^d} u_{m}(x) u_{m}(y)\frac{(\phi_{m}(x)-\phi_{m}(y))^2}{|x-y|^{d+2s}}dxdy dt.
\end{align*}
Using H\"older's inequality we have
\[
	\int_{-1-2^{-m+1}}^{0}\int_{B_2}\eta_m^2u_m|g| dxdt \leq \|g\|_{L^q(Q_4)} \left(\int_{-1-2^{-m+1}}^{0}\int_{\mathbb{R}^d} \left(\eta_{m-1}u_{m}\right)^{q'} dxds\right)^{\frac{1}{q'}}.
\]
}
So we obtain
\begin{align*}
	A_m &\lesssim  \int_{-1-2^{-m+1}}^{0}\int_{0}^{\delta^m}\int_{B_2} z^{1-2s}|\nabla\eta_{m}|^2(u^*_{m})^2dxdzdt\\
& \quad + M_0\int_{-1-2^{-m+1}}^{0}\int_{B_2}|\nabla \phi_{m}|^2u_{m}^2 dxds + \int_{-1-2^{-m+1}}^{0}\int_{B_2}\phi_{m}^2 u_{m}^2dxdt  \\
& \quad + \left(\int_{-1-2^{-m+1}}^{0}\int_{\mathbb{R}^d} \left(\eta_{m-1}u_{m}\right)^{q'} dxds\right)^{\frac{1}{q'}} \\
& \quad + \int_{-1-2^{-m+1}}^{0} \iint_{\mathbb{R}^d\times\mathbb{R}^d}   u_{m}(x) u_{m}(y)\frac{(\phi_{m}(x)-\phi_{m}(y))^2}{|x-y|^{d+2s}}dxdy dt.
\end{align*}

\noindent
It is easy to see that
\[
\mathbf{1}_{\left\{\eta_{m-1}u_m > 0 \right\}} \leq  \mathbf{1}_{\left\{\eta_{m-2}u_{m-1} > \frac{\lambda}{2^{m}}\right\}}.
\]
Hence, on the set $\{\eta_{m-1}u_m > 0\}$, we have
\[
\left(\frac{2^m}{\lambda}\eta_{m-1}u_{m}\right)^{q_1}  \leq \left(\frac{2^m}{\lambda}\eta_{m-2}u_{m-1}\right)^{q_1}  \leq \left(\frac{2^m}{\lambda}\eta_{m-2}u_{m-1}\right)^{q_2},
\]
for any $0 < q_1 < q_2$. So, as the proof of \eqref{z12}, by \eqref{Z6}, one has 
\begin{align}\label{z11}
	A_{m}  \leq &  C^m \left(A_{m-3}^{\theta_1} + A_{m-3}^{\theta_2}\right)\\
	& +  C\int_{-1 - 2^{-m + 1}}^{0}\int_{\mathbb{R}^d}\int_{\mathbb{R}^d} u_{m}(x) u_{m}(y)\frac{(\phi_{m}(x)-\phi_{m}(y))^2}{|x-y|^{d+1}}dxdy dt, \nonumber
\end{align}
 for some $\theta_1,\theta_2>1$. 
\vspace{0.2cm}
Now we estimate the last term in \eqref{z11}. To do this, we note that
\begin{align*}
J_1  & :=	\int_{-1-2^{-m+1}}^{0}\int_{B_{1+2^{-m+1/2}}} \int_{B_{1+2^{-m+1/2}}} u_m(x) u_m(y)\frac{(\phi_m(x)-\phi_m(y))^2}{|x-y|^{d+2s}}dxdydt \\
	& \leq C\int_{-1-2^{-m+1}}^{0}\int_{B_{1+2^{-m+1/2}}} \int_{B_{1+2^{-m+1/2}}}u_m(x) u_m(y)\frac{|\nabla\phi_m(\zeta)|^2}{|x - y|^{d+2s-2}}dxdydt \\
& \leq C2^{2m}\int_{-1-2^{-m+1}}^{0}\int_{B_{1+2^{-m+1/2}}}\int_{B_{1+2^{-m+1/2}}} \frac{\phi_{m-1}(x)u_m(x)\phi_{m-1}(y)u_m(y)}{|x-y|^{d+2s-2}}dxdydt\\
&\leq  C2^{2m}\int_{-1-2^{-m+1}}^{0}||\phi_{m-1}u_m||_{L^2}^2dt.
\end{align*}
So by the same way as above we obtain
\begin{equation}
\label{z13}
	J_1 \leq C\frac{2^{m(d+1)/d}}{\lambda^{2s/(d+1-2s)}}A_{m-3}^{\frac{d+1}{d+1-2s}}.
\end{equation}
Moreover, since $\phi_m(y) = 0$ for all $y \in B^c_{1+2^{-m+1/2}}$ it follows that
\begin{align}
J_2 & = \int_{-1-2^{-m+1}}^{0}\int_{B_{1+2^{-m + 1/2}}^c}\int_{B_{1+2^{-m + 1/2}}} u_m(x) u_m(y)\frac{(\phi_m(x)-\phi_m(y))^2}{|x-y|^{d+2s}}dxdydt  \label{z15} \\ 
& \leq \int_{-1-2^{-m+1}}^{0}\int_{B_{1+2^{-m + 1/2}}^c}\int_{B_{1+2^{-m}}} \frac{\phi_{m-1}(x)u_m(x)u_m(y)}{|x - y|^{d+2s}}dxdydt \nonumber \\
& \leq C\int_{-1-2^{-m+1}}^{0}\int_{B_{1+2^{-m + 1/2}}^c}\frac{u_m(y)}{|y|^{d+1}}dy\int_{B_{1+2^{-m}}}\phi_{m-1}(x)u_m(x)dxdt \nonumber \\
& \leq C\|\phi_{m-1}u_m\|_{L^{1}([-1-2^{-m+1}, 0] \times \mathbb{R}^d)}  \leq C\left(\frac{2^{m}}{\lambda}\right)^{\frac{2(d+1)}{d+1-2s}-1} A_{m-3}^{\frac{d+1}{d+1-2s}}. \nonumber
\end{align}
Here we used \eqref{z01} in the third inequality. Thus, combining \eqref{z11}, \eqref{z13} and  \eqref{z15}  one get 
\begin{align*}
A_{m}  &\leq   C^m \left(A_{m-3}^{\theta_1} + A_{m-3}^{\theta_2}\right)+C(J_1+J_2)\\&\leq  C^m \left(A_{m-3}^{\theta_1} + A_{m-3}^{\theta_2}+A_{m-3}^{\frac{d+1}{d+1-2s}}\right).
\end{align*}
This implies \eqref{z6} for some for some $c_2,\theta_1,\theta_2>1$. The proof is complete. 
\end{proof}

\noindent

\begin{proposition}\label{lem2}  Let $q > (d +1)/2s$. Assume that $g \in L^q(Q_4)$ and the vector field $\mathbf{B}$ with divergence free satisfy following conditions 
	\begin{align}\label{Z10}
		\|\mathbf{B}\|_{L^\infty(-4,0;L^{d/s}(B_2))}\leq M_0~~\text{and}~~ ||g||_{L^q(Q_4)}\leq 1.
	\end{align}
	Let $u$ be a solution of the problem \eqref{eq1} such that 
	\begin{align}\label{Z14}
		u^* \leq 2 ~~ \text{in} ~~ Q^\star_4, ~~~ \int_{B^c_4} \frac{(u(x, t)-2)_+}{|x|^{d+2s}}dx \leq 2 ~~ \text{for} ~~ t \in [-4, 0],
	\end{align}
	\vspace{-0.3cm}
	\begin{align*}
		\mathcal{L}_\omega^{d+2}\left(\left\{(x,z,t)\in Q_4^\star: u^*(x,z,t)\leq 0\right\}\right)\geq \mathcal{L}_\omega^{d+2}\left( Q_4^\star\right)/2.
	\end{align*}
	
	\noindent
	Then we can find a positive constant $C = C(s, d, M_0, \Lambda)$  such that the following statement holds: For each $\varepsilon_1 > 0$ there exists a positive constant $\delta_1 = \delta_1(s, d, M_0, \Lambda, \varepsilon_1)$ so that if
	\begin{equation*}
		\mathcal{L}_\omega^{d+2}\left(\left\{(x,z,t)\in Q_4^\star: 0< u^*(x,s,t)<1 \right\}\right)\leq \delta_1
	\end{equation*}
	then we have 
	\begin{equation}\label{zz06}
	\int_{Q_1}(u-1)_+^2dxdt+\int_{Q_1^\star}z^{1-2s}(u^* - 1)_+^2 dxdzdt\leq C\varepsilon_1^s.
	\end{equation}
	Here $\mathcal{L}_\omega^{d+2}$ is measure in space $\mathbb{R}^{d+2}$ with respect to the weighted measure $\omega dxdz := z^{1-2s}dxdz$. 
\end{proposition}

\vspace{0.1cm}
\qquad In order to prove this lemma we need the following weighted version of De Giorgi's isoperimetric inequality which was proven in \cite{Cons-Wu}. Its unweighted version was given  in \cite{CaVa}).
\begin{proposition}{(\cite[Lemma 3.5]{Cons-Wu})}\label{prop.a1}
	Let $p >  \frac{1-s}{s}$ and let $f$ be a function defined in $B_r^\star$ such that
	\[
		\mathcal{K} := \int_{B_r^\star}z^{1-2s}|\nabla f|^2dxdz < \infty.
	\]
	We denote
	\begin{align*}
		& \mathcal{A} = \{(x, z) \in B_r^\star : f(x,z) \leq 0 \}, \\
		& \mathcal{B} = \{(x, z) \in B_r^\star : f(x,z) \geq 1 \}, \\
		& \mathcal{C} = \{(x, z) \in B_r^\star : 0 < f(x,z) < 1 \},
	\end{align*}
	and $\mathcal{L}_\omega^{d+1}$ is measure in space $\mathbb{R}^{d+1}$ with respect to the weighted measure $\omega dxdz := z^{1-2s}dxdz$. Then we have
	\[
		\mathcal{L}^{d+1}_\omega(\mathcal{A})\mathcal{L}^{d+1}_\omega(\mathcal{B}) \leq Cr^{\varrho}\mathcal{K}^{\frac{1}{2}}\left(\mathcal{L}^{d+1}_\omega(\mathcal{C})\right)^{\frac{1}{2p}},
	\]
	where
	\[
		\varrho = 1 + \frac{1}{2}\left(d + 1 - \frac{p+1}{p-1}(1-2s)\right)\left(1 - \frac{1}{p}\right).
	\]
\end{proposition}

\begin{proof}[Proof of Lemma \ref{lem2}]
	We use the similar arguments as in \cite{CaVa}.
	
	\vspace{0.2cm}
	\qquad It follows from the energy inequality \eqref{energe1} of Proposition \ref{prop2} and assumptions \eqref{Z10}, \eqref{Z14}   and Holder's inequality that there exists a positive constant $C$ such that
	\begin{equation}\label{Z15}
		\int_{Q_4^\star}z^{1-2s}\left|\nabla u^*_+\right|^2 dxdzdt \leq C.
	\end{equation}
	Let $\varepsilon_1 \ll 1$. We put
	\begin{gather*}
	 		\Upsilon =  4\varepsilon_1^{-1}\int_{Q_4^\star}z^{1-2s}\left|\nabla u^*_+\right|^2 dxdzdt, \\
			\mathbb{I}_1 = \left\{t \in [-4, 0] : \int_{B_4^\star}z^{1-2s}\left|\nabla u^*_+\right|^2 dxdz \leq \Upsilon \right\},
	\end{gather*}
	and for every $t \in \mathbb{I}_1$,
	\begin{align*}
	& \mathcal{A}(t) = \Big\{(x, z) \in B_4^\star : u^*(x, z, t) \leq 0 \Big\},\\
	& \mathcal{B}(t) = \Big\{(x, z) \in B_4^\star : u^*(x, z, t) \geq 1 \Big\}, \\
	& \mathcal{C}(t) = \Big\{(x, z) \in B_4^\star : 0 < u^*(x, z, t) < 1 \Big\}.
	\end{align*}
	
	\noindent
	Given a fixed number $p > (1-s)s^{-1}$.
	
	\vspace{0.2cm}\noindent
	{\bf Claim 1.} ~  If $\mathbb{I}$ is the set defined by
	\[
		\mathbb{I} = \mathbb{I}_1 \cap \mathbb{I}_2 ~~ \text{with} ~~ \mathbb{I}_2 := \Big\{t \in [-4, 0] : \big(\mathcal{L}_\omega^{d+1}(\mathcal{C}(t))\big)^{\frac{1}{2p}} \leq \varepsilon_1^{1 + \frac{p}{s}}\Big\}.
	\]
	Then we have 
	\begin{equation}\label{Z17}
		\mathcal{L}^1\big([-4, 0] \setminus \mathbb{I}\big) \leq \frac{\varepsilon_1}{2},
	\end{equation}
	provided that we choose $\delta_1$ such that $\delta_1 < \varepsilon_1^{2p(1+p/s) + 2}$. This is because the following inequalities hold true 
	\begin{equation}\label{zz01}
		\mathcal{L}^1\left([-4, 0] \setminus \mathbb{I}_1 \right) \leq \Upsilon^{-1}\int_{Q_4^\star}z^{1-2s}\left|\nabla u^*_+\right|^2 dxdzdt \leq \frac{\varepsilon_1}{4},
	\end{equation}
	\begin{align}\label{zz02}
	\mathcal{L}^1\left([-4, 0] \setminus \mathbb{I}_2\right) &\leq \varepsilon_1^{-2p(1+p/s)} \mathcal{L}_\omega^{d+2}\left(\{(x, z, t) : 0 < u^* <1 \}\right) \leq \varepsilon_1^{-2p(1+p/s)}\delta_1  \leq \frac{\varepsilon_1}{4},
	\end{align}
	thanks to Chebyshev's inequality.
	
	\vspace{0.2cm}\noindent
	{\bf Claim 2.} ~ If $t \in \mathbb{I}$ such that $\mathcal{L}_\omega^{d+1}(\mathcal{A}(t)) \geq 1/4$ then
	\begin{align}\label{zz04}
		& \int_{B_4^\star}z^{1-2s} (u^*_+)^2dxdz  \leq \varepsilon_1^\frac{p}{s} ~~ \text{and} ~~ \int_{B_4}u_+^2dx \leq \varepsilon_1^\frac{1-s}{2s}.
	\end{align}
	Indeed, we use first the Proposition \ref{prop.a1} to obtain
	\[
	\mathcal{L}_\omega^{d+1}(\mathcal{B}(t)) \leq C\Upsilon^{\frac{1}{2}}\left(\mathcal{L}_\omega^{d+1}(\mathcal{C}(t))\right)^{\frac{1}{2p}} 	(\mathcal{L}_\omega^{d+1}(\mathcal{A}(t)))^{-1}\leq C\varepsilon_1^{1/2 + p/s}.
	\]
	Hence, since $u^* \leq 2$ in $Q_4^*$ we have
	\begin{equation}
	\label{Z11}
		\int_{B_4^\star}z^{1-2s}(u^*_+)^2dxdz \leq 4\left(\mathcal{L}_\omega^{d+1}(\mathcal{B}(t)) + \mathcal{L}_\omega^{d+1}(\mathcal{C}(t))\right) \leq C\varepsilon_1^{\frac{1}{2}+\frac{p}{s}} \leq  \varepsilon_1^\frac{p}{s}.
	\end{equation}

	On the other hand, for any $z$, it is clear that
	\[
		\int_{B_4}u_+^2 dx = \int_{B_4}(u^*_+(x,z))^2dx - 2\int_0^z\int_{B_4}u^*_+\partial_z u^*dxd\bar{z}.
	\]
	Then by integrating in $z$ on $[0, 1]$ this implies
	\begin{equation}\label{Z9}
		\int_{B_4}u_+^2 dx \leq \int_{B^\star_4}(u^*_+)^2dxdz + 2\int_{B^\star_4}|u^*_+|\,|\partial_z u^*|dxdz.
	\end{equation}
	From H\"older's inequality and the fact that $u^* \leq 2$ we have
	\begin{align*}
		\int_{B^\star_4}(u^*_+)^2dxdz &\leq 2^{2-\sigma}\int_{B^\star_4}z^{\frac{(2s-1)\sigma}{2}}z^{\frac{(1-2s)\sigma}{2}}(u^*_+)^{\sigma}dxdz \\
		&\leq 2^{2-\sigma}\left(\int_{B^\star_4}z^{\frac{(2s-1)\sigma}{2-\sigma}} \right)^{\frac{2-\sigma}{2}}\left(\int_{B_4^\star}z^{1-2s}(u^*_+)^2dxdz\right)^{\frac{\sigma}{2}},
	\end{align*}
	where $\sigma < (1-s)^{-1}$. Since $p\sigma > 1$ it follows that
\begin{equation}\label{Z12}
	\int_{B^\star_4}(u^*_+)^2dxdz \leq C\varepsilon_1^{\frac{p\sigma}{2s}}\leq \varepsilon_1^{\frac{1}{2s}} .
\end{equation}
	Similarly we have from H\"older's inequality 
	\begin{align*}
		&2\int_{B^\star_4}|u^*_+|\,|\partial_z u^*|dxdz  \leq 2^{2-1/p}\int_{B^\star_4}z^{\frac{1-2s}{2p}}|u^*_+|^{\frac{1}{p}}\, z^{\frac{1-2s}{2}}|\partial_z u^*|\, z^{-\frac{(1-2s)(p+1)}{2p}}dxdz \\
		& \leq C\left(\int_{B_4^\star}z^{1-2s}|\partial_z u^*|^2dxdz\right)^\frac{1}{2}\left(\int_{B_4^\star}z^{1-2s}|u^*|^2dxdz\right)^\frac{1}{2p}\left(\int_{B_4^\star} z^{-\frac{(1-2s)(p+1)}{p-1}}dxdz\right)^{\frac{p-1}{2p}}\\ & \overset{\eqref{Z11}}\leq C\Upsilon^{\frac{1}{2}}\varepsilon_1^{\frac{1}{4p} + \frac{1}{2s}},
	\end{align*}
	provided that $\frac{(1-2s)(p+1)}{p-1}<1$. 
	This implies
	\begin{equation}\label{Z13}
		\int_{B^\star_4}|u^*_+|\,|\partial_z u^*|dxdz  \leq C\varepsilon_1^{-\frac{1}{2}+\frac{1}{4p} + \frac{1}{2s}}.
	\end{equation}
	And therefore we obtain from \eqref{Z9},\eqref{Z12} and \eqref{Z13} that
	\[
		\int_{B_4}u_+^2dx \leq \varepsilon_1^{\frac{1}{2s}}+C\varepsilon_1^{-\frac{1}{2}+\frac{1}{4p} + \frac{1}{2s}}\leq  \varepsilon_1^{\frac{1-s}{2s}}.
	\]
So, we get \eqref{zz04} for $\varepsilon_1>0$ small enoough. 
	
	\vspace{0.2cm}\noindent
	{\bf Claim 3.} ~ If $t \in \mathbb{I} \cap [-1, 0] $ then $\displaystyle \mathcal{L}_\omega^{d+1}(\mathcal{A}(t)) \geq \frac{1}{4}$.
	
	\noindent
	Indeed, since $$\mathcal{L}_\omega^{d+2}\left(\left\{(x,z,t)\in Q_4^\star: u^*(x,z,t)\leq 0\right\}\right)\geq \mathcal{L}_\omega^{d+2}\left( Q_4^\star\right)/2,$$ there exists $-4\leq t_0 \leq -1$ such that $\mathcal{L}_\omega^{d+1}(\mathcal{A}(t_0)) \geq 1/4$. Thus, by using Claim 2 it follows that
	\[
		\int_{B_4} u_+^2(x, t_0)dx \leq \varepsilon_1^{\frac{1-s}{2s}}.
	\]
	Let $r > 0$ be a small number. Combining this estimate and the energy inequality \eqref{energe1} of Proposition \ref{prop2} and assumptions \eqref{Z10}, \eqref{Z14}   and Holder's inequality  we deduce that, for any $t \geq t_0$,
	\begin{align*}
	\int_{B_1}u_+^2(x, t)dx  \notag & \leq  \int_{B_4}u_+^2(x, t_0)dx + C(t-t_0+(t-t_0)^{\frac{q-1}{q}}), \nonumber
	\end{align*}
	So for $t - t_0 \leq \delta^* = \big(\min\{1, (10^{10}C)^{-1}\}\big)^{\frac{10q}{q-1}}$, we have
	\begin{equation}\label{Z16}
		\int_{B_1}u_+^2(x, t)dx \leq 10^{-9}, ~~ \text{for all} ~~ t_0 \leq t \leq t_0 + \delta^\star.
	\end{equation}
	for $\varepsilon_1>0$ small enough. It is also noted that $\delta^\star$ do not depend on $\varepsilon_1$. Hence we can suppose  that $\varepsilon_1 \ll \delta^\star$.
	
	\vspace{0.2cm}
	\quad Next, we have
	\begin{align*}
		u^*_+ & = u_+ + \int_0^z\partial_{\bar{z}} u^*_+d\bar{z}.
	\end{align*}
	This implies that
	\begin{align*}
		z^{1-2s}(u^*_+)^2 & \leq  2z^{1-2s}u_+^2  + 2z^{1-2s}\left(\int_0^z\partial_{\bar{z}}u^*_+d\bar{z}\right)^2 \\
		& \leq 2z^{1-2s}u_+^2  + \frac{z}{s}\int_0^z \bar{z}^{1-2s}|\nabla u^*_+|^2d\bar{z}.
	\end{align*}
	Thus, for given $t \in \mathbb{I}$ such that $t_0 \leq t \leq t_0 + \delta^\star$, by integrating in $(x, z)$ on $B_1 \times [0, \varepsilon_1^{1/s}]$ we get
	\begin{align*}
		\int_0^{\varepsilon_1^{1/s}}\int_{B_1}z^{1-2s}(u^*_+)^2dxdz  & \leq  \frac{\varepsilon_1^{2(1-s)/s}}{1-s}\int_{B_1}u_+^2dx  + \frac{\varepsilon_1^{2/s}}{2s}\int_0^{\varepsilon_1^{1/s}}\int_{B_1} z^{1-2s}|\nabla u^*_+|^2dz \\
		& \overset{\eqref{Z16},\eqref{Z15}}\leq 10^{-8}\varepsilon_1^{2(1-s)/s}+ C(s)\varepsilon_1^{2/s-1} \leq 10^{-7}\varepsilon_1^{2(1-s)/s}.
	\end{align*}
	Using Chebyshev's inequality again, it follows that
	\begin{align*}
	\mathcal{L}_\omega^{d + 1}(\{x \in B_1, z \in [0, \varepsilon_1^{1/s}] : u^*_+\geq 1\})  \leq 10^{-7}\varepsilon_1^{2(1-s)/s}.
	\end{align*}
	Since $\mathcal{L}_\omega^{d+1}(\mathcal{C}(t)) \leq \varepsilon_1^{2p(1+\frac{p}{s})}$ we have 
$$
		\mathcal{L}_\omega^{d+1}(\mathcal{A}(t))  \geq \frac{\varepsilon_1^{2(1-s)/s}}{2(1-s)}\mathcal{L}^{d}(B_1) - 10^{-7}\varepsilon_1^{2(1-s)/s} - \varepsilon_1^{2p(1+\frac{p}{s})} \geq \frac{\varepsilon_1^{2(1-s)/s}}{4}.$$
	for $\varepsilon_1>0$ small enough. \\
	Using again Proposition \ref{prop.a1} we get
	\[
		\mathcal{L}_\omega^{d+1}(\mathcal{B}(t)) \leq C\frac{\Upsilon^{\frac{1}{2}}\mathcal{L}_\omega^{d+1}(\mathcal{C}(t))^{\frac{1}{2p}}}{\mathcal{L}_\omega^{d+1}(\mathcal{A}(t))} \leq C\varepsilon_1^{\frac{1}{2}+\frac{p}{s} - 2(1-s)/s} \leq C\sqrt{\varepsilon_1}.
	\]
	So 
	\[
		\mathcal{L}_\omega^{d+1}(\mathcal{A}(t)) \geq \mathcal{L}_\omega^{d+1}(B_4^\star)  - \mathcal{L}_\omega^{d+1}(\mathcal{B}(t)) - \mathcal{L}_\omega^{d+1}(\mathcal{C}(t)) \geq \frac{1}{4}.
	\]
	In summary, we have proved that $ \mathcal{L}_\omega^{d+1}(\mathcal{A}(t)) \geq 1/4$ for every $t \in \mathbb{I}\cap [t_0, t_0 + \delta^*]$. It is easy to find from \eqref{Z17} that on $[t_0 + \delta^\star/2, t_0 + \delta^\star]$ there exists $t_1 \in \mathbb{I}$. Hence we can find an increasing sequence $\{t_n\}_{n=1}^\infty$ such that
	\begin{align}
		&  t_0 + \frac{n\delta^\star}{2} \leq t_n \leq 0, \label{zz05} \\
		& \mathcal{L}_\omega^{d+1}(\mathcal{A}(t)) \geq \frac{1}{4} ~~ \text{if} ~~ t \in \Big[t_n, t_n + \delta^\star\Big] \cap \mathbb{I} \supset [t_n, t_{n+1}] \cap \mathbb{I}.
	\end{align}
	Therefore we conclude that $\mathcal{L}_\omega^{d+1}(\mathcal{A}(t)) \geq 1/4$ on $\mathbb{I} \cap [-1, 0]$.
	
	\vspace{0.2cm}\noindent
	{\bf The estimate \eqref{zz06} holds true.}
	
	\vspace{0.2cm}
	From Claim 3 and Proposition \ref{prop.a1}  we imply that, for each $t \in \mathbb{I} \cap [-1, 0]$, $$\mathcal{L}_\omega^{d+1}(\mathcal{B}(t)) \leq 4C\varepsilon_1^{\frac{1}{2}+\frac{p}{s}} \leq \frac{\varepsilon_1}{16}$$ which leads to
	\[
		\mathcal{L}_\omega^{d+2}(\{(x, z, t) \in Q_1^\star : u^* \geq 1 \}) \leq \frac{\varepsilon_1}{16} + 	\mathcal{L}^1\big([-4, 0] \setminus \mathbb{I}\big)  \overset{\eqref{Z17}}\leq \varepsilon_1.
	\]
	Hence, by $(u^* - 1)_+ \leq 1$,
\begin{equation}\label{Z18}
	\int_{Q_1^\star}z^{1-2s}(u^* - 1)_+^2dxdzdt \leq \varepsilon_1.
\end{equation}

	For every fixed $x$ and $t$, we have
	\[
		u = u^* -\int_0^z\partial_z u^*d\bar{z}.
	\]
	Hence, for any $z$,
	\begin{align*}
		z^{1-2s}(u-1)_+^2 & \leq 2z^{1-2s}(u^* - 1)_+^2 + 2z^{1-2s}\left(\int_0^z|\nabla(u^*|d\bar{z}\right)^2 \\
		& \leq 2z^{1-2s}(u^* - 1)_+^2 + z\int_0^z\bar{z}^{1-2s}|\nabla(u^*|^2d\bar{z}.
	\end{align*}
	By taking the average in $z$ on $[0, \sqrt{\varepsilon_1}]$ we get
	\[
		\frac{\varepsilon_1^{\frac{1-2s}{2}}}{2(1-s)}(u-1)_+^2  \leq \frac{2}{\sqrt{\varepsilon_1}}\int_0^{\sqrt{\varepsilon_1}}z^{1-2s}(u^* - 1)_+^2 dz + \sqrt{\varepsilon_1}\int_0^{\sqrt{\varepsilon_1}}z^{1-2s}|\nabla(u^*|^2dz.
	\]
	Therefore, by Integrating with respect to $(x, t)$ on $B_1 \times [-1, 0]$ it follows from \eqref{Z15} and \eqref{Z18} that
	\[
		\frac{\varepsilon_1^{\frac{1-2s}{2}}}{2(1-s)}	\int_{Q_1}(u - 1)_+^2dxdt \leq 2\sqrt{\varepsilon_1}.
	\]
	So, 
	\begin{equation*}
	\int_{Q_1}(u - 1)_+^2dxdt \leq C\sqrt{\varepsilon_1}.
	\end{equation*}
	The results follows from this and \eqref{Z18}. The proof is complete. 
\end{proof}

\begin{proof}[Proof of Lemma \ref{osclem1}] 
	Put $\varepsilon_1 = \displaystyle \hat{\varepsilon}_0^2/16C^2$  with $C$ is the constant defined in Lemma \ref{lem2}. Let us take the constants $\lambda$ and $\delta_1$ be associated to $\varepsilon_0$ as in Lemma \ref{lem1} and to $\varepsilon_1$ as in Lemma \eqref{lem2}, respectively.
	
	\vspace{0.1cm}
	\quad Define $\bar{u}_0 = u$ and  
	$$
		\bar{u}_k=2(\bar{u}_{k-1}-1), ~~ \text{for every $k\in \mathbb{N}$, $k\leq K_0=[1 + |B_1|(4(1-s)\delta_1)^{-1}]$}.
	$$
\noindent
	Then, for every $k$, 
	\begin{itemize}
		\item[$\diamond$] $\bar{u}_k=2^k(u-2)+2$, 
		\item[$\diamond$] $\bar{u}_k$ is solution \eqref{eq1} with data $2^kg$ and satisfies $(\bar{u}_k)^* \leq 2$. Moreover, we also have
		\begin{equation*}
			\int_{B^c_1} \frac{(\bar{u}_k(x)-2)_+}{|x|^{d+2s}}dxdt \leq 1,
		\end{equation*}
		and
		\begin{equation*}
			\mathcal{L}_\omega^{d+2}\left(\left\{(x,z,t)\in Q_1^\star: (\bar{u}_k)^*(x,z,t)\leq 0\right\}\right)\geq  \frac{|B_1|}{4(1-s)},
		\end{equation*}
		where  we choose $\varepsilon_2>0$ such that $2^{K_0}\varepsilon_2= 1.$
	\end{itemize}
	
	 \noindent
	 Now we consider two following cases:
	 
	 \vspace{0.2cm}\noindent
	 {\bf Case 1}: ~ There is $k_0 \leq K_0$ such that
	 \[
	 		\mathcal{L}_\omega^{d+2}\left(\left\{(x,z,t)\in Q_1^\star: 0< (\bar{u}_{k_0})^*(x,z,t) < 1\right\}\right)\leq  \delta_1.
	 \]
	By applying Lemma \eqref{lem2} for solution $\bar{u}_{k_0 }$ with data $2^{k_0}g$, and then Lemma \ref{lem1} for solution $\bar{u}_{k_0 + 1}$ with data $2^{k_0+1}g$, we obtain 
	$$
		\bar{u}_{k_0+1} \leq 2-\lambda ~~ \text{on} ~~ Q_{\frac{1}{8}}.
	$$ 
	Thus, 
	\begin{equation*}
		u\leq 2-2^{-k_0-1}\lambda\leq 2-2^{-K_0}\lambda ~~\text{in}~~Q_{1/8}.
	\end{equation*}

	 \vspace{0.2cm}\noindent
	{\bf Case 2}: ~ For all $k\leq K_0$, we have 
		$$
			\mathcal{L}_\omega^{d+2}\left(\left\{(x,z,t)\in Q_1^\star: 0< (\bar{u}_{k})^*(x,z,t)<1\right\}\right) > \delta_1.
		$$ 
		Then for every $k\leq K_0$,
	\begin{align*}
		\mathcal{L}_\omega^{d+2}\left(\left\{(x,z,t) \in Q_1^\star: (\bar{u}_{k})^* < 0\right\}\right) & =	\mathcal{L}_\omega^{d+2}\left(\left\{(x,z,t) \in Q_1^\star: (\bar{u}_{k-1})^* < 1\right\}\right) \\
		& \geq \delta_1 +\mathcal{L}_\omega^{d+2}\left(\left\{(x,z,t) \in Q_1^\star: (\bar{u}_{k-1})^* \leq 0\right\}\right).
	\end{align*}
	Thus, we get 
	$$	
		\mathcal{L}_\omega^{d+2}\left(\left\{(x,z,t) \in Q_1^\star: (\bar{u}_{K_0})^* < 0\right\}\right)\geq \frac{|B_1|}{4(1-s)},
	$$ 
	and it leads to $(\bar{u}_{K_0})^* \leq 0$ almost everywhere,  which means $$u^* \leq 2-2^{-K_0}.$$ 
	Then, as the proof of \cite[Propsition 9]{CaVa}, we get \eqref{z7}. 
\end{proof}

\setcounter{equation}{0}
\section{Proof of the main results}

 \qquad This section is devoted proving the main results, Theorem \ref{eq1} and Corollary \ref{XX1}. Firstly, we shall prove the following Theorem which leads to Theorem \ref{eq1} as a direct consequence.
\begin{theorem}\label{XX} Assume that the vector field $\mathbf{B}$ satisfies the conditions 
	\begin{align*}
	\sup\limits_{x\in \mathbb{R}^{d}, \,\, t>-2}\left|\fint_{x + B_1} \mathbf{B}(y, t)dy \right| \leq M_1 \quad \text{and} \quad  \|\mathbf{B}\|_{L^\infty((-2,
	\infty); \mathbb{X}^s)} \leq M_2,
	\end{align*}
	for some $M_1, M_2 \geq 1$. There exists $\varepsilon_0>0$  such that if we have
	\begin{equation}
	\|u(-2)\|_{L^{2}(\mathbb{R}^{d})}+	\|g\|_{L^1(\mathbb{R}^{d}\times (-2,\infty))} + \|g\|_{L^{q}(\mathbb{R}^{d}\times (-2,\infty)}\leq \varepsilon_0, 
	\end{equation}
	then the following estimate holds true
		\begin{equation}\label{z14'}
	\sup_{|x_1 - x_2| + |t_1 - t_2|^{1/(2s)} \leq \rho_0, t_1>-1/2}\frac{|u(x_1, t_1) - u(x_2, t_2)|}{\big(|x_1 - x_2| + |t_1 - t_2|^{1/(2s)}\big)^\alpha} \leq  C,
	\end{equation}
	where $\rho_0=C(M_2)/M_1\leq 1/10$.
\end{theorem}
\begin{proof} By Remark \eqref{ZA}, one has 
	\begin{equation}
		\sup_{t>-2}\int_{\mathbb{R}^d}u(t)^2dx+\int_{-2}^{\infty}\int_{\mathbb{R}^{d}}|(-\Delta)^{s/2} u|^2 dxdt + \|u\|_{L^{\infty}(\mathbb{R}^{d}\times (-1,\infty))}^2\leq C\varepsilon_0^2.
	\end{equation}
	In particular, 
	\begin{equation}\label{q1}
		|u(x, t)|\leq 1~~\text{in}~\mathbb{R}^d\times (-1,0].
	\end{equation}
	for $\varepsilon_0>0$ small enough. \\
	
	Now we put $b_{-1} \equiv \mathbf{B}$ and, for $k \geq 0$,
	\[
	 	b_k(y, t) = \sigma^{2s-1}b_{k-1}(\sigma y + \sigma^{2s}x_k(t)), \sigma^{2s} t) - \sigma^{2s-1}\dot{x}_k(t), 	\]
	where $\sigma < 1$ is a constant that will be chosen later, $x_k$ is the solution to problem
	\begin{equation*}
		\dot x_k(t)=\frac{1}{\mathcal{L}^{d}(B_4)}\int_{B_4}b_{k-1}(\sigma y + \sigma^{2s}x_k(t)), \sigma^{2s} t) dy, ~~ x_k(0) = 0.
	\end{equation*}
	We also define the sequences $\{g_k\}_{k=-1}^\infty$ and $\{F_k\}_{k=-1}^\infty$ defined by recursively as follows:
	\begin{itemize}
		\item $\displaystyle g_{-1} \equiv g, ~~  F_{-1}  \equiv u$ and
		\item  for every $k \geq 0$, 
		\begin{align*}
				g_k(y, t) = \widehat{\lambda}\sigma^{2s} g_{k-1}\Big(\sigma y + \sigma^{2s}x_k(t), \sigma^{2s} t\Big),
		\end{align*}
		\begin{align*}
		& F_k(y, t) = \frac{8}{8-\lambda^\star}\left[F_{k-1}\Big(\sigma y + \sigma^{2s}x_k(t), \sigma^{2s} t\Big) + \lambda^\star/4\right],  ~~ \text{or} \\
		\\
		& F_k(y, t) = \frac{8}{8-\lambda^\star}\left[F_{k-1}\Big(\sigma y + \sigma^{2s}x_k(t), \sigma^{2s} t\Big)-\lambda^\star/4\right],
		\end{align*}
	\end{itemize}
where $\lambda^\star\in (0,10^{-10})$ is the constant defined in Lemma \ref{osclem1}.  

\vspace{0.2cm}\noindent
Clearly, for any $k= 0,1, 2,...$, the function $F_k$ satisfies 
\begin{equation*}
	\partial_tF_k +b_k\nabla F_k +\mathcal{L}_{k,t} F_k=g_k,
\end{equation*}
where
\[
	\mathcal{L}_{j,t} f(y) = \int_{\mathbb{R}^d} (f(y) - f(y'))\mathbf{K}_{j,t}(y, y')dy',
\]
with
\[
	\mathbf{K}_{j,t}(y, y') = \sigma^{d+2s}\mathbf{K}_{j-1,t}\left(\sigma y -\sigma^{2s}x_j(t), \sigma y' - \sigma^{2s}x_j(t) \right).
\]
It is necessary to note that, for all $y, y' \in \mathbb{R}^d$, we have $\mathbf{K}_{j,t}(y, y') = \mathbf{K}_{j,t}(y', y)$ and
\[
	\frac{\Lambda^{-1}}{|y - y'|^{d+2s}} \leq \mathbf{K}_{j,t}(y, y') \leq \frac{\Lambda}{|y-y'|^{d+2s}},
\]

\quad On the other hand, for $k \geq 0$, we also have 
\begin{align*}
	\operatorname{div}b_k = 0,~~~\int_{B_4}b_k(y, t)dy=0, ~~ \|b_k(\cdot, t)\|_{\mathbb{X}_{s}} = \|\mathbf{B}(\cdot, \sigma^{2s(k+1)}t)\|_{\mathbb{X}_s}.
\end{align*}
In particular, it leads to
\begin{align*}
	\|b_k\|_{L^\infty((-1,0];L^{q}(B_4))}\leq CM_2, ~ \text{for all} ~ k\geq 0,  q\geq 1.
\end{align*}
Moreover, one can check that
\begin{equation}\label{abc1}
	g_k(y,t) = \left(\frac{8}{8-\lambda^\star}\right)^{k+1}\sigma^{2s(k+1)}g\Big(\sigma^{k+1}y+\sum_{m=0}^{k}\sigma^{k+2s-m}x_m(t), \sigma^{2s(k+1)} t\Big),
\end{equation}
and 
\begin{align}\label{bs1}
	F_k(y_1, t_1) - F_k(y_2, t_2) = & \left(\frac{8}{8-\lambda^\star}\right)^{k+1} u\Big(\sigma^{k+1}y_1 + \sum_{m=0}^{k}\sigma^{k+2s-m}x_m(t_1),\sigma^{2s(k+1)}t_1\Big) \\
	 & -  \left(\frac{8}{8-\lambda^\star}\right)^{k+1} u\Big(\sigma^{k+1}y_2 + \sum_{m=0}^{k}\sigma^{k+2s-m}x_m(t_2),\sigma^{2s(k+1)}t_2\Big). \nonumber
\end{align}

\noindent
\textbf{Step 1: (Boundedness of functions $x_k(t)$)} ~ We shall prove that, for any $k \geq 0$, the following estimate holds true
\begin{align}\label{abc2} 
	\sup_{t\in [-1,0]} |x_k(t)|+\sup_{t\in [-1,0]} |\dot x_k(t)| \lesssim  \Theta\big(k, s, \sigma, M_1, M_2\big),
\end{align}
where the constant $\Theta = \Theta\big(k, s, \sigma, M_1, M_2\big)$ given by
\begin{equation}\label{abc3}
	\Theta = \begin{cases}
		M_1 + M_2|\log\sigma|, & \text{if} ~~ s = \frac{1}{2}, ~ k = 0, \\
		M_2 + M_2|\log\sigma|, & \text{if} ~~ s = \frac{1}{2}, ~ k > 0, \\
		M_1 + M_2, & \text{if} ~~ s < \frac{1}{2}, ~ k = 0, \\
		M_2^{\frac{1}{2s}}, & \text{if} ~~ s < \frac{1}{2}, ~ k > 0.
	\end{cases}
\end{equation}

\quad First we consider the case $0 < s < 1/2$. Since $\mathbf{B} \in L^\infty\big(\mathbb{R}; C^{1-2s}(\mathbb{R}^d)\big)$ one has
\begin{multline*}
	|\dot{x}_0(t)|   = \left|\fint_{B_4}\mathbf{B}(\sigma y + \sigma^{2s}x_0(t), \sigma^{2s}t)dy \right| \\
	 = \left|\fint_{\sigma^{2s}x_0(t) + B_{4\sigma}} \mathbf{B}(y, \sigma^{2s} t)dy\right|  \leq  \left|\fint_{\sigma^{2s}x_0(t) + B_{1}} \mathbf{B}(y, \sigma^{2s} t)dy\right| \\
	 \qquad + \left|\fint_{\sigma^{2s}x_0(t) + B_{4\sigma}} \mathbf{B}(y, \sigma^{2s} t)dy - \fint_{\sigma^{2s}x_0(t) + B_{1}} \mathbf{B}(y, \sigma^{2s} t)dy \right|.
\end{multline*}
This implies $$|\dot{x}_0(t)|  \lesssim M_1 + (1 + \sigma^{1-2s})M_2.$$ Combining this and $x_k(0) = 0$ we obtain \eqref{abc2} for case $k=0$ and $s<1/2$.

\vspace{0.2cm}
\noindent For any $k \geq 1$, by using the fact $\displaystyle \int_{B_4}b_{k-1}(y, \sigma^{2s}t)dy = 0$, it follows that
\begin{align*}
	|\dot{x}_k(t)| & = \left|\fint_{\sigma^{2s}x_k(t) + B_{4\sigma}} b_{k-1}(y, \sigma^{2s} t)dy -  \fint_{B_{4}} b_{k-1}(y, \sigma^{2s} t)dy\right| 
						& \lesssim \left(1 +  |x_k(t)|^{1-2s}  \right)M_2.
\end{align*}
Since $x_k(0) = 0$ we get
\begin{align*}
	\sup_{t \in [-1, 0]}|x_k(t)| \lesssim M_2 + M_2\big(\sup_{t \in [-1, 0]}|x_k(t)|\big)^{1-2s}
\end{align*}
which implies $\sup_{t \in [-1, 0]}|x_k(t)|  \lesssim M_2^{1/2s}$. This will give \eqref{abc2} for case $k\geq 1$ and $s<1/2$.

\medskip 
\vspace{0.2cm}
\quad Next we consider the case $s = 1/2$. For $\delta < 1$ and $\xi \in \R^d$, we have
\begin{multline}
	\left|\fint_{\xi + B_\delta}B(y, t)dy\right|  \leq  \left|\fint_{\xi + B_1}B(y, t)dy\right| + \left|\fint_{\xi + B_\delta}B(y, t)dy - \fint_{\xi + B_{2^{[\log_2\delta]}}}B(y, t)dy\right| \label{bs.a02} \\
	 \qquad + \sum\limits_{j = 1 + [\log_2\delta]}^{0}\left|\fint_{\xi + B_{2^j}}B(y, t)dy - \fint_{\xi + B_{2^{j-1}}}B(y, t)dy\right|\\
	 \leq M_1 + 2^d\left(1 - \log_2\delta\right)M_2 \lesssim M_1 + |\log_2\delta|M_2. 
\end{multline}
It follows from \eqref{bs.a02} that
\begin{align*}
|\dot{x}_0(t)|  & = \left|\fint_{B_4}\mathbf{B}(\sigma y + \sigma x_0(t), \sigma t)dy \right| = \left|\fint_{\sigma x_0(t) + B_{4\sigma}} \mathbf{B}(y, \sigma  t)dy\right| \lesssim M_1 + |\log\sigma|M_2  .
\end{align*}
Combining this and $x_0(0) = 0$ we get \eqref{abc2}. Now for $k \geq 1$, since $\displaystyle \fint_{B_4}b_{k-1}(y,t)dy = 0$ it follows
\begin{align*}
	\left|\fint_{\sigma x_k(t) + B_{4\sigma}} b_{k-1}(y, \sigma t)dy\right| \leq 	\left|\fint_{\sigma x_k(t) + B_{4\sigma}} b_{k-1}(y, \sigma t)dy - \fint_{B_{4\sigma}} b_{k-1}(y, \sigma t)dy\right| \\
	+ \left|\fint_{B_{4\sigma}} b_{k-1}(y, \sigma t)dy - \fint_{B_{4}} b_{k-1}(y, \sigma t)dy\right|.
\end{align*}
On the other hand, by using \cite[Corollary 5.1.10]{Auscher} we have
\begin{align*}
	\left|\fint_{\sigma x_k(t) + B_{4\sigma}} b_{k-1}(y, \sigma t)dy - \fint_{B_{4\sigma}} b_{k-1}(y, \sigma t)dy\right| \lesssim \, \log_2(2 + |x_k(t)|),
\end{align*}
and, similarly as \eqref{bs.a02},
\begin{align*}
	\left|\fint_{B_{4\sigma}} b_{k-1}(y, \sigma t)dy - \fint_{B_{4}} b_{k-1}(y, \sigma t)dy\right| \leq (1 - \log_2\sigma)M_2.
\end{align*}
Hence we can obtain
\begin{align*}
	|\dot{x}_k(t)|  & = \left|\fint_{\sigma x_k(t) + B_{4\sigma}} b_{k-1}(y, \sigma t)dy\right| \lesssim \big[|\log\sigma| + \log(2 + |x_k(t)|)\big]M_2.
\end{align*}
Thanks to $x_k(0) = 0$ we deduce that
$$
\sup_{t \in [-1,0]} |x_k(t)| \lesssim \Big\{\big|\log\sigma\big| + \log(2 + \sup_{t \in [-1,0]} |x_k(t)|) \Big\}M_2 
$$
which implies
\begin{align*}
	\sup_{t\in [-1,0]} |x_k(t)| \lesssim |\log\sigma|M_2 + M_2\log M_2.
\end{align*}
So one gets
\begin{equation*}
	\sup_{t\in [-1,0]} |x_k(t)|+\sup_{t\in [-1,0]} |\dot x_k(t)| \lesssim |\log\sigma|M_2 + M_2\log M_2.
\end{equation*}

\vspace{0.2cm}\noindent
\textbf{Step 2: (Control the functions $F_k$)}  For any $k \geq 0$ we shall prove that
\begin{equation}\label{Z21}
	|F_{k}(x,t)| \leq 2 + c_0\lambda^\star\big(|\sigma^{1-2s} x + x_k(t)|^{\frac{2s}{3}}-1\big)_+ ~~\text{in}~\mathbb{R}^d\times (-1,0],
\end{equation}
and 
\begin{equation}\label{Z22}
|F_{k}(x,t)| \leq 2~~\text{in}~Q_{1/16},
\end{equation}
where $c_0$ will be choosen later. To do this we use the induction argument.

\vspace{0.2cm}\noindent
Let $(x, t) \in \mathbb{R}^d\times (-1,0]$. It follows from \eqref{q1} that
\begin{equation*}\label{Z20}
	|F_{0}(x, t)|\leq \frac{8}{8-\lambda^\star}\Big(1 + \frac{\lambda^\star}{4}\Big) \leq 2 \leq 2 +c_0\lambda^\star(|\sigma^{1-2s} x + x_0(t)|^{\frac{2s}{3}}-1)_+.
\end{equation*}
\vspace{0.2cm}\noindent
Similarly, by the definition of $F_1$ we imply
\begin{align*}
	|F_{1}(t,x)| & \leq  \frac{8}{8-\lambda^\star}\left\{\frac{8}{8-\lambda^\star}\Big(1+\frac{\lambda^\star}{4}\Big)+\frac{\lambda^\star}{4}\right\} \leq 2 \\
	& \leq  2 +c_0\lambda^\star\big(|\sigma^{1-2s} x + x_1(t)|^{\frac{2s}{3}}-1\big)_+.
\end{align*}
Assume that \eqref{Z21} and \eqref{Z22} hold for any $k\leq k_1$ for $k_1\geq 1$. We will need to prove
\begin{equation}\label{Z21e}
	|F_{k_1+1}(x,t)| \leq 2 + c_0\lambda^\star\big(|\sigma^{1-2s} x + x_{k_1+1}(t)|^{\frac{2s}{3}}-1\big)_+.
\end{equation}
and 
\begin{equation}\label{Z22b}
|F_{k_1+1}(x,t)| \leq 2~~\text{in}~Q_{1/16}.
\end{equation}
Indeed, it follows from Step 2 and the induction assumption that 
\begin{align*}
	|\mathbf{P}(F_{k_1})(x,z)|&\leq \int_{\mathbb{R}^d}P(x-y,z)\left[2 +c_0\lambda^\star\big(|\sigma y + \sigma^{2s}x_{k_1}(t)|^{\frac{2s}{3}}-1\big)_+\right] dy\\
	&\leq  2+ c_0\lambda^\star \int_{\mathbb{R}^d}P(x-y,z)|\sigma^{1-2s} y + x_{k_1}(t)|^{\frac{2s}{3}} dy\\
	& \leq  2+ c_0\lambda^\star \int_{\mathbb{R}^d}P(x-y,z)\big(|y|^{\frac{2s}{3}} + |x_{k_1}(t)|^\frac{2s}{3} \big)dy.
\end{align*}

By \eqref{abc2}-\eqref{abc3} and the definition of Poisson kernel it follows that
\begin{align*}
		|\mathbf{P}(F_{k_1})(x,z)| & \leq  2+ c_0\lambda^\star\left(C+  \Theta^\frac{2s}{3}\right) \leq 2+ 2c_0\lambda^\star C(M_2)|\log\sigma|.
\end{align*}
Moreover,  we also have, 
\begin{equation}
\int_{B^c_1} \frac{(F_{k_1}(x,t)-2)_+}{|x|^{d+2s}}dx \leq c_0\lambda^\star C(M_2)|\log\sigma|.
\end{equation}
Since $\|g_k\|_{L^q(Q_4)} \leq \|g\|_{L^q(Q_4)}$  we can apply  Lemma \eqref{osclem1}  with $c_0\lambda^\star C(M_2)|\log\sigma|\leq \varepsilon_2$ ( here $\varepsilon_2$ is the constant in Lemma \eqref{osclem1}) to the function $$(x, t) \mapsto \frac{F_{k_1}(x,t)}{1 + c_0\lambda^\star C(M_2)|\log\sigma|}$$ and get that 
\begin{equation*}
	F_{k_1} \leq (2-\lambda^\star)(1 + c_0\lambda^\star C(M_2)|\log\sigma|)~~\text{in}~~Q_{1/16},
\end{equation*} 
or 
\begin{equation*}
	- F_{k_1}\leq (2-\lambda^\star)(1 + c_0\lambda^\star C(M_2)|\log\sigma|)~~\text{in}~~Q_{1/16}.
\end{equation*} 
Taking $c_0 =\varepsilon_2 \left(4C(M_2)|\log\sigma|\right)^{-1}$ yields 
\begin{equation*}
	F_{k_1} \leq 2-\frac{\lambda^\star}{2} ~~\text{in}~~Q_{1/16},
\end{equation*} 
or 
\begin{equation*}
 - F_{k_1}\leq 2 - \frac{\lambda^\star}{2} ~~\text{in}~~Q_{1/16}.
\end{equation*}
In the case $F_{k_1} \leq 2-\lambda^\star/2 ~~ \text{in} ~ Q_{1/16}$ we shall set 
 \begin{align*}
	  F_{k_1+1}(y, t) =\frac{8}{8-\lambda^\star} \left(F_{k_1}\big(\sigma y + \sigma^{2s} x_{k_1+1}(t)), \sigma^{2s} t) + \frac{\lambda^\star}{4}\right),
 \end{align*}
	 and if $- F_{k_1} \leq 2-\lambda^\star/2 ~~ \text{in} ~ Q_{1/16}$, then we set 
\begin{align*}
	F_{k_1+1}(y, t) = \frac{8}{8-\lambda^\star} \left(F_{k_1}\big(\sigma y + \sigma^{2s} x_{k_1+1}(t), \sigma^{2s} t) - \frac{\lambda^\star}{4}\right).
\end{align*}
Clearly,
\begin{equation*}
	|F_{k_1+1}(y, t)|	\leq 2 ~~\text{if}~ |\sigma^{1-2s} y + x_{k_1+1}(t)| \leq \frac{1}{16\sigma^{2s}}.
\end{equation*}
Since \eqref{Z21} holds for $k=k_1$, 
\begin{equation*}
	\left|F_{k_1+1}(y, t)\right| \leq  \frac{8}{8-\lambda^\star} \left(2 + c_0\lambda^\star\big(|\sigma^{1-2s} (\sigma y + \sigma^{2s} x_{k_1+1}(t)) + x_{k_1}(\sigma^{2s}t)|^{\frac{2s}{3}}-1\big)_++ \frac{\lambda^\star}{4}\right).
\end{equation*}
Thus, it remains to show that, for any $t \in [-1, 0]$ and $y \in \mathbb{R}^d$ such that 
\[
|z| > \frac{1}{16\sigma^{2s}},
\]
we have
\begin{equation}\label{Z19}
	G(z, t) \leq  2 +c_0\lambda^\star \big(|z|^{2s/3}-1\big)_+.
\end{equation}
Here 
\[
	G(z, t) := \frac{8}{8-\lambda^\star} \left\{2 +c_0\lambda^\star\left(\big|\sigma z + x_{k_1}(\sigma^{2s} t)\big|^{2s/3}-1\right)_+ + \lambda^\star/4\right\},
\]
with $z = \sigma^{1-2s} y +x_{k_1+1}(t)$. Indeed, if $ |z| >  \frac{1}{16\sigma^{2s}}$ then
\begin{align*}
	G(y, t) & = 2+\frac{4\lambda^\star}{8-\lambda^\star}\left\{1+2c_0 \left(\big|\sigma z + x_{k_1}(\sigma^{2s} t)\big|^{2s/3}-1\right)_+\right\} \\ 
	& \overset{\eqref{abc2}}\leq  2 + \frac{4\lambda^\star}{8-\lambda^\star}\left\{1 + 2c_0\left(|\sigma z|^{2s/3} +  \Theta^{\frac{2s}{3}}\right)\right\}.
\end{align*}
This implies from \eqref{abc3} by choosing $\sigma = \sigma(M_2) \ll 1$ that
\[
	G(y, t) \leq 2 +  \frac{1}{2}c_0\lambda^\star|z|^{2s/3} \leq 2+c_0\lambda^\star \big(|z|^{2s/3}-1\big)_+
\]
This implies \eqref{Z19}.

\vspace{0.2cm}\noindent
\textbf{Step 3:} By \eqref{Z22b} in Step 2, we find that 
\begin{equation}
\operatorname{osc}_{Q_{\frac{1}{16}}} F_k \leq 4, ~~ \text{for} ~ k \geq 0.
\end{equation}
So, thanks to \eqref{bs1} one has for any $(y_1,t_1),(y_2,t_2)\in Q_{1/16}$
\begin{align}\label{Z23}
u\Big(\sigma^{k+1}y_1 +\sum_{m=0}^{k}\sigma^{k+2s-m}x_m(t_1),\sigma^{2s(k+1)}t_1\Big) \hspace{3cm} {\color{white} aaa} \\  -u\Big(\sigma^{k+1}y_2+\sum_{m=0}^{k}\sigma^{k+2s-m}x_m(t_2),\sigma^{2s(k+1)}t_2\Big) \hspace{1.5cm} {\color{white} aaa}  \nonumber \\
 \leq 4 \left(1 - \frac{\lambda^\star}{8}\right)^{k+1}. \nonumber
\end{align}

\vspace{0.2cm}\noindent
If $s = 1/2$ the from \eqref{abc2}  we imply that, if $-\sigma^{k+1} < t  < 0$, 
\begin{align*}
	\left|\sum_{m=0}^{k}\sigma^{k+2s-m}x_m(t)\right| & \leq \sigma^{k+1}\left\{ \sigma^{k+2s}\big(M_1+|\log(\sigma)|\big) + \sum_{m=1}^{k}\sigma^{k+2s-m} |\log\sigma|\right\}C(M_2) \\	
	&  \leq \frac{\sigma^{k+1}}{32},
\end{align*}
provided that $\sigma \ll_{M_2} 1$. Here we note that this is also true when $0 < s \leq 1/2$. Hence  it follows from \eqref{Z23} that
$$
	u(y_1,t_1)-u(y_2,t_2)\leq 4 \left(1 - \frac{\lambda^\star}{8}\right)^{k+1},
$$
for any $y_1,y_2\in B_{\sigma^{k+1}/32}$ and any $t_1, t_2\in \big(-\sigma^{2s(k+1)}/16,0\big)$. By choosing 
\begin{equation}
	\sigma = \sigma(M_2) \ll 1, \quad k \geq k_0 := \left[\frac{\log M_1}{|\log\sigma|}\right] + 1,
\end{equation}
we get for any $\rho < \rho_0 :=  \sigma^{k_0+1}/32\sim C(M_2)/M_1$ 
\begin{equation*} 
\operatorname{osc}_{B_{\rho}\times \big(-\rho^{2s},\, 0\big)} u\leq C\rho^{\alpha}, 
\end{equation*}
where
\[
	\alpha=\frac{\big|\log\big(1-\frac{\lambda^\star}{8}\big)\big|}{\big|\log\sigma|}.
\]
This gives 
\[
	\sup\limits_{|x_1 - x_2| + |t_1 - t_2|^{1/(2s)}< \rho_0, t_1>-1/2} \frac{\big|u(x_1, t_1) - u(x_2, t_2)\big|}{(|x_1 - x_2| + |t_1 - t_2|^{1/(2s)})^\alpha } \leq C.
\]

\vspace{0.2cm}
The proof of our theorem is completed.
\end{proof}
\begin{proof}[ Proof of Corollary \ref{XX1}] By Remark \ref{ZA}, one has 
	\begin{equation*}
	||u(1)||_{L^\infty(\mathbb{R}^d)}\leq C||u(0)||_{L^2(\mathbb{R}^d)}.
	\end{equation*}
 Since $u_\lambda(t,x)=\lambda^{-2s}u(\lambda^{2s} t,\lambda x)$ ($\lambda>0$) is also solution to \eqref{eq1} with drift term $$\mathbf{B}_\lambda(t,x)=\lambda^{1-2s}\mathbf{B}(\lambda^{2s}t,\lambda x),$$ it follows that
	\begin{equation*}
	||u_\lambda(1)||_{L^\infty(\mathbb{R}^d)}\leq C||u_
	\lambda(0)||_{L^2(\mathbb{R}^d)}.
	\end{equation*}
	This implies, for $t=\lambda^{2s}$,
	\begin{equation*}
	||u(t)||_{L^\infty(\mathbb{R}^d)}\leq Ct^{-\frac{d}{4s}} ||u(0)||_{L^2(\mathbb{R}^d)},
	\end{equation*}
	and we obtain \eqref{XX2}. \\
	
	By Theorem \ref{CaVa}, we have
	\begin{equation}
	\sup_{|x_1 - x_2| + |t_1 - t_2|^{1/(2s)} \leq \rho_0, \,\, t_1,t_2>1}\frac{|u(x_1, t_1) - u(x_2, t_2)|}{\big(|x_1 - x_2| + |t_1 - t_2|^{1/(2s)}\big)^\alpha}\leq C(M_2) ||u(0)||_{L^2(\mathbb{R}^d)},
	\end{equation}
	where $\rho_0=C_2(M_1)/M_1$.  We apply this estimate for $u_\lambda(t,x)=\lambda^{-2s}u(\lambda^{2s} t,\lambda x)$ and $$\mathbf{B}_\lambda(t,x)=\lambda^{1-2s}\mathbf{B}(\lambda^{2s}t,\lambda x)$$ to derive 
	\begin{equation}
	\sup_{|x_1 - x_2| + |t_1 - t_2|^{1/(2s)} \leq c_0, \, t_1,t_2>1}\frac{|u(\lambda^{2s} t_1,\lambda x_1) - u(\lambda^{2s}t_2,\lambda x_2)|}{\big(|x_1 - x_2| + |t_1 - t_2|^{1/(2s)}\big)^\alpha}\leq C(M_2) \lambda^{-d/2}||u(0)||_{L^2(\mathbb{R}^d)},
	\end{equation}
	with 
	\begin{equation}
	c_0=\frac{C(M_2,M_3)}{ \lambda^{1-2s}\sup\limits_{(x,t)\in \mathbb{R}^{d+1}_+}\left|\fint_{x + B_\lambda} \mathbf{B}(y, t)dy \right| +1}.
	\end{equation}
	
	\noindent
	{\bf Case $s=1/2$:} Since 
	\begin{equation}
	\sup\limits_{(x,t)\in \mathbb{R}^{d+1}_+}\left|\fint_{x + B_\lambda} \mathbf{B}(y, t)dy \right| \leq C (\log(|\lambda|)+1)M_2+M_1,
	\end{equation}
	it implies that
	\begin{equation}
	c_0 \geq \frac{C(M_2,M_1)}{|\log(|\lambda)|+1},
	\end{equation}
	and hence
	\begin{equation*}
	\sup_{|x_1 - x_2| + |t_1 - t_2| \leq \frac{C(M_2,M_1)\lambda}{|\log(|\lambda)|+1}, \,\, t_1,t_2>\lambda}\frac{|u(t_1, x_1) - u(t_2, x_2)|}{\big(|x_1 - x_2| + |t_1 - t_2|\big)^\alpha}\leq C(M_2) \lambda^{-d/2-\alpha}||u(0)||_{L^2(\mathbb{R}^d)}.
	\end{equation*}
	This implies \eqref{XX3}.
	
	\vspace{0.2cm}\noindent
	{\bf Case $0<s<1/2$:} Since  
	\begin{equation}
	c_0\geq \frac{C(M_3)}{\lambda^{1-2s}+1},
	\end{equation}
	we can deduce that
	\begin{multline*}
	\sup_{|x_1 - x_2| + |t_1 - t_2|^{1/(2s)}  \leq C(M_3)\min\{\lambda,\lambda^{2s}\},~~ t_1,t_2>\lambda^{2s}}\frac{|u(t_1, x_1) - u(t_2, x_2)|}{\big(|x_1 - x_2| + |t_1 - t_2|^{1/(2s)}\big)^\alpha} \\
	 \qquad\leq C(M_2) \lambda^{-d/2-\alpha}||u(0)||_{L^2(\mathbb{R}^d)}.
	\end{multline*}
	This implies \eqref{XX4}. The proof is complete.
\end{proof}

\setcounter{equation}{0}

\section*{Appendix A}

\qquad We state here a technical lemma which implies the crucial lemma (Lemma \ref{crucialLemma}.)

\medskip \medskip \noindent
{\bf Lemma A.} {\it 
For any $a_1,a_2\in \mathbb{R}$ and $b_1,b_2\in \mathbb{R}_+$, there holds
\begin{align}\label{Z4}
(b_1a_1-b_2a_2)^2+|a_1||a_2|(b_1-b_2)^2\sim  (a_1-a_2)(b_1^2a_1-b_2^2a_2)+  C|a_1||a_2|(b_1-b_2)^2
\end{align}
and 
\begin{align} \label{Z5}
&(b_1a_1-b_2a_2)	(b_1a_1^+-b_2a_2^+)+(|a_1|a_2^++a_1^+|a_2|)(b_1-b_2)^2\\&~~\sim  (a_1-a_2)(b_1^2a_1^+-b_2^2a_2^+)+  C(|a_1|a_2^++a_1^+|a_2|)(b_1-b_2)^2\nonumber
\end{align}
for some constant $C>0$.
}

\begin{proof} In order to prove \eqref{Z4}, without loss of generality, we can assume that $|a_1| \geq |a_2|$. Then we have    
	\begin{align*}
	 (a_1 - a_2)(b_1^2a_1 - b_2^2a_2) & = (a_1 - a_2)(b_1^2a_1 - b_1b_2a_2 + b_1b_2a_2 - b_2^2a_2) \\
	 & = (b_1a_1 - b_1a_2)(b_1a_1 - b_2a_2) +(b_2a_1 - b_2a_2)a_2(b_1 - b_2) \\
	 & = (b_1a_1-b_2a_2)^2+(b_2-b_1)a_2(b_1a_1-b_2a_2) \\
	 & \quad \,\, + (b_1a_1-b_2a_2)a_2(b_1-b_2)-a_1a_2(b_1-b_2)^2.
	\end{align*}
	Thanks to Holder's inequality, we get that 
	\begin{align*}
	\left|	(a_1-a_2)(b_1^2a_1-b_2^2a_2)-(b_1a_1-b_2a_2)^2\right|\leq \frac{1}{10}(b_1a_1-b_2a_2)^2+C(|a_2|^2+|a_1||a_2|)(b_1-b_2)^2.
	\end{align*}
	This implies \eqref{Z4}.  
	
	\vspace{0.2cm}
	Since \eqref{Z4} holds true, we only need to check \eqref{Z5} in case $a_1\geq 0>a_2$. Indeed, one has 
	\begin{align*}
	(a_1-a_2)(b_1^2a_1^+-b_2^2a_2^+) & = (b_1a_1-b_1a_2)b_1a_1^+ \\
	& = (b_1a_1-b_2a_2)b_1a_1^++(b_2-b_1)a_2b_1a_1^+\\ 
	& = (b_1a_1-b_2a_2)b_1a_1^+-(b_2-b_1)^2a_2a_1^++(b_2-b_1)a_2a_1^+b_2.
	\end{align*}
	Hence
	\begin{align*}
	&	\left|(a_1-a_2)(b_1^2a_1^+-b_2^2a_2^+)-(b_1a_1-b_2a_2)b_1a_1^+\right|\\&~~~~~~~~\leq  (b_2-b_1)^2|a_2|a_1^++\min\{|(b_2-b_1)a_2a_1^+b_2|,|(b_2-b_1)a_2b_1a_1^+|\}\\&~~~~~~~~\leq C
	|a_2|a_1^+(b_1-b_2)^2+\frac{1}{10}|a_2|a_1^+ b_1b_2
	\\&~~~~~~~~\leq C
	|a_2|a_1^+(b_1-b_2)^2+\frac{1}{10}(b_1a_1-b_2a_2)b_1a_1^+.
	\end{align*}
	It follows \eqref{Z5} in case $a_1\geq 0>a_2$.
\end{proof}

\begin{proof}[Proof of Lemma \ref{crucialLemma}] We have 
	\begin{align*}
	&	\int_{\mathbb{R}^d} h^2 f^+ \mathcal{L}_t f dx=\frac{1}{2}\int_{\mathbb{R}^d}\int_{\mathbb{R}^d} \left(h(x)^2 f^+(x)-h(y)^2 f^+(y  \right)\left(f(x)-f(y)\right)\mathbf{K}_t(x,y)dxdy,\\&
	\int_{\mathbb{R}^d} h^2 f^+ (-\Delta)^s f dx=c_{d,s}\int_{\mathbb{R}^d}\int_{\mathbb{R}^d} \left(h(x)^2 f^+(x)-h(y)^2 f^+(y  \right)\left(f(x)-f(y)\right)\frac{dxdy}{|x-y|^{d+2s}}.
	\end{align*}
	Thus, the result follows from Lemma \ref{Z4}. The proof is complete. 
\end{proof}

\section*{Acknowledgements} This research is partly funded by University of Economics Ho Chi Minh City, Vietnam.

\bibliographystyle{alpha} 
\bibliography{biblio}

\end{document}